\documentclass[twoside,11pt]{article}

\usepackage[preprint]{jmlr2e}
\usepackage{algorithm,algpseudocode}
\usepackage{amsbsy,amsmath,amsfonts}
\usepackage{appendix}
\usepackage{bookmark}
\usepackage{booktabs}
\usepackage{xcolor}

\def\vec{\mathop{\rm vec}\nolimits}

\def\dist{\mathop{\rm dist}\nolimits}
\def\prox{\mathop{\rm prox}\nolimits}
\def\argmin{\mathop{\rm argmin}\nolimits}
\def\dom{\mathop{\rm dom}\nolimits}
\def\amp{\mathop{\;\:}\nolimits}
\newcommand{\ba}{\boldsymbol{a}}
\newcommand{\bb}{\boldsymbol{b}}
\newcommand{\bc}{\boldsymbol{c}}
\newcommand{\bd}{\boldsymbol{d}}
\newcommand{\be}{\boldsymbol{e}}

\newcommand{\bg}{\boldsymbol{g}}

\newcommand{\bp}{\boldsymbol{p}}
\newcommand{\bq}{\boldsymbol{q}}
\newcommand{\br}{\boldsymbol{r}}
\newcommand{\bs}{\boldsymbol{s}}

\newcommand{\bu}{\boldsymbol{u}}
\newcommand{\bv}{\boldsymbol{v}}
\newcommand{\bw}{\boldsymbol{w}}
\newcommand{\bx}{\boldsymbol{x}}
\newcommand{\by}{\boldsymbol{y}}
\newcommand{\bz}{\boldsymbol{z}}
\newcommand{\bA}{\boldsymbol{A}}
\newcommand{\bB}{\boldsymbol{B}}
\newcommand{\bC}{\boldsymbol{C}}
\newcommand{\bD}{\boldsymbol{D}}

\newcommand{\bG}{\boldsymbol{G}}
\newcommand{\bH}{\boldsymbol{H}}
\newcommand{\bI}{\boldsymbol{I}}

\newcommand{\bM}{\boldsymbol{M}}
\newcommand{\bN}{\boldsymbol{N}}

\newcommand{\bQ}{\boldsymbol{Q}}

\newcommand{\bS}{\boldsymbol{S}}
\newcommand{\bT}{\boldsymbol{T}}
\newcommand{\bU}{\boldsymbol{U}}
\newcommand{\bV}{\boldsymbol{V}}
\newcommand{\bW}{\boldsymbol{W}}
\newcommand{\bX}{\boldsymbol{X}}
\newcommand{\bY}{\boldsymbol{Y}}
\newcommand{\bZ}{\boldsymbol{Z}}

\newcommand{\bbeta}{\boldsymbol{\beta}}

\newcommand{\blambda}{\boldsymbol{\lambda}}
\newcommand{\bmu}{\boldsymbol{\mu}}

\newcommand{\bsigma}{\boldsymbol{\sigma}}
\newcommand{\btheta}{\boldsymbol{\theta}}

\newcommand{\bxi}{\boldsymbol{\xi}}

\newcommand{\bXi}{\boldsymbol{\Xi}}

\newcommand{\bSigma}{\boldsymbol{\Sigma}}

\newcommand{\bzero}{\boldsymbol{0}}

\newtheorem{secprop}{Proposition}[section]

\jmlrheading{1}{2022}{1-48}{4/00}{10/00}{XXXX}{Alfonso Landeros, Oscar Hernan Madrid Padilla, Hua Zhou, and Kenneth Lange}

\ShortHeadings{Extensions to the Proximal Distance Method of Constrained Optimization}{Landeros, Padilla, Zhou, and Lange}
\firstpageno{1}

\begin{document}

\title{Extensions to the Proximal Distance Method of Constrained Optimization}

\author{\name Alfonso Landeros \email alanderos@ucla.edu \\
       \addr Department of Computational Medicine\\
       University of California, Los Angeles\\
       CA 90095-1596, USA
       \AND
       \name Oscar Hernan Madrid Padilla \email oscar.madrid@stat.ucla.edu \\
       \addr Department of Statistics\\
       University of California, Los Angeles\\
	   CA 90095-1596, USA
	   \AND
	   \name Hua Zhou \email huazhou@ucla.edu \\
       \addr Departments of Biostatistics and Computational Medicine\\
       University of California, Los Angeles\\
	   CA 90095-1596, USA
	   \AND
	   \name Kenneth Lange \email klange@ucla.edu \\
       \addr Departments of Computational Medicine, Human Genetics, and Statistics\\
       University of California, Los Angeles\\
       CA 90095-1596, USA
}
\editor{TBA}

\maketitle

\begin{abstract}
The current paper studies the problem of minimizing a loss $f(\bx)$ subject to constraints of the form $\bD\bx \in S$, where $S$ is a closed set, convex or not, and $\bD$ is a matrix that fuses the parameters. Fusion constraints can capture smoothness, sparsity, or more general constraint patterns. To tackle this generic class of problems, we combine the Beltrami-Courant penalty method of optimization with the proximal distance principle. The latter is driven by minimization of penalized objectives $f(\bx)+\frac{\rho}{2}\text{dist}(\bD\bx,S)^2$ involving large tuning constants $\rho$ and the squared Euclidean distance of $\bD\bx$ from $S$. The next iterate $\bx_{n+1}$ of the corresponding proximal distance algorithm is constructed from the current iterate $\bx_n$ by minimizing the majorizing surrogate function $f(\bx)+\frac{\rho}{2}\|\bD\bx-\mathcal{P}_{S}(\bD\bx_n)\|^2$. For fixed $\rho$ and a subanalytic loss $f(\bx)$ and a subanalytic constraint set $S$, we prove convergence to a stationary point. Under stronger assumptions, we provide convergence rates and demonstrate linear local convergence.
We also construct a steepest descent (SD) variant to avoid costly linear system solves. To benchmark our algorithms, we compare their results to those delivered by the alternating direction method of multipliers (ADMM). Our extensive numerical tests include problems on metric projection, convex regression, convex clustering, total variation image denoising, and projection of a matrix to a good condition number. These experiments demonstrate the superior speed and acceptable accuracy of our steepest variant on high-dimensional problems. Julia  code to replicate all of our experiments can be found at \url{https://github.com/alanderos91/ProximalDistanceAlgorithms.jl}
\end{abstract}


\section{Introduction}

The generic problem of minimizing a continuous function $f(\bx)$ over a closed set $S$ of $\mathbb{R}^p$ can be attacked by a combination of the penalty method and distance majorization. The classical penalty method seeks the solution of a penalized version $h_\rho(\bx)=f(\bx)+\rho q(\bx)$ of $f(\bx)$, where the penalty $q(\bx)$ is nonnegative and 0 precisely when $\bx \in S$. If one follows the solution vector $\bx_\rho$ as $\rho$ tends to $\infty$, then in the limit one recovers the constrained solution \citep{beltrami1970algorithmic,courant1943variational}. The function 
\begin{eqnarray*}
q(\bx) & = & \frac{1}{2} \dist(\bx,S)^2 
\amp = \amp \frac{1}{2}\min_{\by \in S} \|\bx-\by\|^2
\end{eqnarray*}
is one of the most fruitful penalties in this setting. Our previous research for solving this penalized minimization problem has focused on an MM (majorization-minimization) algorithm based on distance majorization \citep{chi2014distance,keys2019proximal}. In distance majorization one constructs the surrogate function
\begin{eqnarray*}
g_\rho(\bx \mid \bx_n) & = & f(\bx)+\frac{\rho}{2}
\|\bx-\mathcal{P}(\bx_n)\|^2 
\end{eqnarray*}
using the Euclidean projection $\mathcal{P}(\bx_n)$ of the current iterate $\bx_n$ onto $S$. The minimum of the surrogate occurs at the proximal point
\begin{eqnarray}
\bx_{n+1} & = & \prox_{\rho^{-1}f}[\mathcal{P}(\bx_n)].
\label{prox_dist_update}
\end{eqnarray}
According to the MM principle, this choice of $\bx_{n+1}$ decreases $g_\rho(\bx \mid \bx_n)$ and hence the objective $h_\rho(\bx)$ as well.
As we note in our previous {\it JMLR} paper \citep{keys2019proximal}, the update (\ref{prox_dist_update}) reduces to the classical proximal gradient method when $S$ is convex \citep{parikh2014proximal}.

We have named this iterative scheme the proximal distance algorithm \citep{keys2019proximal,lange2016mm}. It enjoys several virtues. First, it allows one to exploit the extensive body of results on proximal maps and projections. Second, it does not demand that the constraint set $S$ be convex. If $S$ is merely closed, then the map $\mathcal{P}(\bx)$ may be multivalued, and one must choose a representative element from the projection $\mathcal{P}(\bx_n)$. Third, the algorithm does not require the objective function $f(\bx)$ to be differentiable. Fourth, the algorithm dispenses with the chore of choosing a step length. Fifth, if sparsity is desirable, then the sparsity level can be directly specified rather than implicitly determined by the tuning parameter of the lasso or other penalty. 

Traditional penalty methods have been criticized for their numerical instability.  This hazard is mitigated in the proximal distance algorithm by its reliance on proximal maps, which usually are highly accurate.  The major defect of the proximal distance algorithm is slow convergence. This can be ameliorated by Nesterov acceleration \citep{nesterov2013introductory}. There is also the question of how fast one should send $\rho$ to $\infty$. Although no optimal schedule is known, simple numerical experiments usually yield a good choice. Finally, soft constraints can be achieved by stopping the steady increase of $\rho$ at a finite value.

\subsection{Proposed Framework}
\label{sec:extensions}

This simple version of distance majorization can be generalized in various ways. For instance, it can be expanded to multiple constraint sets. In practice, at most two constraint sets usually suffice. Another generalization is to replace the constraint $\bx \in S$ by the constraint $\bD \bx \in S$, where $\bD$ is a compatible matrix. Again, the original case $\bD=\bI$ is allowed. By analogy with the fused lasso of \cite{tibshirani2005sparsity}, we will call the matrix $\bD$ a fusion matrix. This paper is devoted to the study of the general problem of minimizing a differentiable function $f(\bx)$ subject to $r$ fused constraints $\bD_i \bx \in S_i$. We will approach this problem by extending the proximal distance method. For a fixed penalty constant $\rho$, the objective function and its MM surrogate now become 
\begin{eqnarray*}
h_\rho(\bx) & = & f(\bx) + \frac{\rho}{2} \sum_{i=1}^r \dist(\bD_i\bx,S_i)^2 \\ 
g_\rho(\bx \mid  \bx_n) & = & f(\bx) + \frac{\rho}{2} \sum_{i=1}^r \|\bD_i\bx - \mathcal{P}_i(\bD_i \bx_n)\|^2,
\end{eqnarray*}
where $\mathcal{P}_i(\by)$ denotes the projection of $\by$ onto $S_i$. Any or all of the fusion matrices $\bD_i$ can be the identity $\bI$.

Fortunately, we can simplify the problem by defining $S$ to be the Cartesian product $\prod_{i=1}^r S_i$ and $\bD$ to be the stacked matrix 
\begin{eqnarray*}
\bD & = & \begin{pmatrix} \bD_1 \\ \vdots \\ \bD_r \end{pmatrix} .
\end{eqnarray*}
Our objective and surrogate then revert to the less complicated forms
\begin{eqnarray}
h_\rho(\bx) & = & f(\bx) + \frac{\rho}{2} \dist(\bD\bx,S)^2 \label{obj_fun}\\
g_\rho(\bx \mid  \bx_n) & = & f(\bx) + \frac{\rho}{2} \|\bD\bx - \mathcal{P}(\bD \bx_n)\|^2,
\label{surrogate_fun}
\end{eqnarray}
where $\mathcal{P}(\bx)$ is the Cartesian product of the projections $\mathcal{P}_i(\bx)$. Note that all closed sets $S_i$ with simple projections, including sparsity sets, are fair game.

\subsection{Contributions}

In the framework described above, we summarize the contributions of the current paper.

\begin{description}

\item[(a)] Section \ref{sec:methods} describes different solution algorithms for minimizing the penalized loss $h_{\rho}(\bx)$. Our first algorithm is based on Newton's method applied to the surrogate $g_{\rho}(\bx \mid \bx_n)$. For some important problems, Newton's method reduces to least squares.  Our second method is a steepest descent algorithm on $g_{\rho}(\bx \mid \bx_{n})$ tailored to high dimensional problems.

\item[(b)] For a sufficiently large $\rho$, we show that when $f(\bx)$ and $S$ are convex and $f(\bx)$ possesses a unique minimum point $\by \in \bD^{-1}(S)$, the penalized loss $h_{\rho}(\bx)$ attains its minimum value. This is the content of Proposition \ref{propositiona}. Similarly, Proposition \ref{propositionb} shows that the surrogate $g_{\rho}(\bx \mid \bx_{n})$ also attains its minimum.  
    
\item[(c)] If in addition $f(\bx)$ is differentiable, then Proposition \ref{propositionc} demonstrates that the MM iterates $\bx_n$ for minimizing $h_{\rho}(\bx)$ satisfy
\begin{eqnarray*}
 h_{\rho}(\bx_n) - h_{\rho}(\bz_{\rho}) & = & O\left( \frac{\rho}{n}  \right),
\end{eqnarray*}
where $\bz_{\rho}$ minimizes $h_{\rho}(\bx)$. If $f(\bx)$ is also $L$-smooth and $\mu$-strongly convex, then Proposition \ref{proposition0} shows that $\bz_{\rho}$ is unique and the iterates $\bx_n$ converge to $\bz_{\rho}$ at a linear rate.

\item[(d)] More generally, Proposition \ref{MMconvergence} shows
that the iterates $\bx_n$ of a generic MM algorithm for minimizing a coercive subanalytic function $h(\bx)$ with a good surrogate converge to a stationary point. Our objectives and their surrogates fall into this category.

\item[(e)] Finally, we discuss a competing alternating direction method of multipliers (ADMM) algorithm and note its constituent updates. Our extensive numerical experiments compare the 
two proximal distance algorithms to ADMM. We find that proximal distance algorithms are competitive with and often superior to ADMM in terms of accuracy and running time.

\end{description}

\section{Different Solution Algorithms}
\label{sec:methods}

Unless $f(\bx)$ is a convex quadratic, exact minimization of the surrogate $g_\rho(\bx \mid \bx_n)$ is likely infeasible. 
As we have already mentioned, to reduce the objective $h_\rho(\bx)$ in \eqref{obj_fun}, it suffices to reduce the surrogate \eqref{surrogate_fun}. For the latter task, we recommend Newton's method on small and intermediate-sized problems and steepest descent on large problems. The exact nature of these generic methods are problem dependent.
The following section provides a high-level overview of each strategy and we defer details on our later numerical experiments to the appendices.

\subsection{Newton's Method and Least Squares}

Unfortunately, the proximal operator $\prox_{\rho^{-1}f}(\by)$ is no longer relevant in calculating the MM update $\bx_{n+1}$. When $f(\bx)$ is smooth, Newton's method for the surrogate $g_\rho(\bx \mid \bx_n)$ employs the update
\begin{eqnarray*}
\bx_{n+1} & \! = \! & \bx_n - \Big[\bH_n + \rho\bD^t\bD\Big]^{-1}
\Big\{\nabla f(\bx_n)+\rho \bD^t[\bD\bx_n - \mathcal{P}(\bD\bx_n)]\Big\},
\end{eqnarray*}
where $\bH_n=d^2f(\bx_n)$ is the Hessian. To enforce the descent property, it is often prudent to substitute a positive definite approximation  $\bH_n$ for  $d^2f(\bx_n)$. In statistical applications, the expected information matrix is a natural substitute. It is also crucial to retain as much curvature information on $f(\bx)$ as possible. Newton's method has two drawbacks. First, it is necessary to compute and store $d^2f(\bx_n)$. This is mitigated in statistical applications by the substitution just mentioned. Second, there is the necessity of solving a large linear system. Fortunately, the matrix $\bH_n + \rho \bD^t \bD$ is often well-conditioned, for example, when $\bD$ has full column rank and $\bD^t\bD$ is positive definite. The method of conjugate gradients can be called on to solve the linear system in this ideal circumstance.

To reduce the condition number of the matrix $\bH_n + \rho \bD^t\bD$ even further, one can sometimes rephrase the Newton step as iteratively reweighted least squares. For instance, in a generalized linear model, the gradient $\nabla f(\bx)$ and the expected information $\bH$ can be written as
\begin{eqnarray*}
\nabla f(\bx) & = &  - \bZ^t \bW^{1/2}\br \quad \text{and} \quad
\bH \amp = \amp \bZ^t\bW \bZ,
\end{eqnarray*}
where $\br$ is a vector of standardized residuals, $\bZ$ is a design matrix, and $\bW$ is a diagonal matrix of case weights \citep{lange2010numerical,nelder1972generalized}. The Newton step is now equivalent to minimizing the least squares criterion
\begin{eqnarray*}
\frac{1}{2}\bx^*\bH_n\bx-\nabla f(\bx_n)^*\bx & = &
\|\bW_n^{1/2} \bZ \bx-\bW_n^{-1/2} \nabla f(\bx_n)\|^2 \\
& = & \left\|\begin{pmatrix} \bW_n^{1/2}\bZ \\
\sqrt{\rho}\bD \end{pmatrix}\bx 
-\begin{pmatrix} \bW_n^{1/2}\bZ\bx_n+\br_n\\
\sqrt{\rho}\mathcal{P}(\bD\bx_n)  \end{pmatrix}
\right\|^2.
\end{eqnarray*}
In this context a version of the conjugate gradient algorithm adapted to least squares is attractive. The algorithms LSQR \citep{paige1982lsqr} and LSMR \citep{fong2011lsmr} perform well when the design is sparse or ill conditioning is an issue.

\subsection{Proximal Distance by Steepest Descent}
\label{sec:sd}

In high-dimensional optimization problems, gradient descent is typically employed to avoid matrix inversion. Determination of an appropriate step length is now a primary concern. In the presence of  fusion constraints $\bD \bx \in S$ and a convex quadratic loss $f(\bx)=\frac{1}{2}\bx^t\bA\bx+\bb^t\bx$, the gradient of the proximal distance objective at $\bx_n$ amounts to
\begin{eqnarray*}
\bv_n & = &  \bA_n\bx+\bb+\rho \bD^t[\bD\bx_n-\mathcal{P}(\bD\bx_n)].
\end{eqnarray*}
For the steepest descent update $\bx_{n+1}= \bx_n-t_n \bv_n$, one can show that the optimal step length is  
\begin{eqnarray*}
t_n & = & \frac{\|\bv_n\|^2}{\bv_n^t \bA \bv_n + \rho \|\bD\bv_n\|^2}. 
\end{eqnarray*}
This update obeys the descent property and avoids matrix inversion.  One can also substitute a local convex quadratic approximation around $\bx_n$ for $f(\bx)$. If the approximation majorizes $f(\bx)$, then the descent property is preserved. In the failure of majorization, the safeguard of step halving is trivial to implement.

In addition to Nesterov acceleration, gradient descent can be accelerated by the subspace MM technique \citep{chouzenoux2010majorize}. Let $\bG_n$ be the matrix with $k$ columns determined by the $k$ most current gradients of the objective $h_\rho(\bx)$, including $\nabla h_\rho(\bx_n)$. Generalizing our previous assumption, suppose $f(\bx)$ has a quadratic surrogate with Hessian $\bH_n$ at $\bx_n$. Overall we get the quadratic surrogate
\begin{eqnarray*}
q_\rho(\bx \mid \bx_n) & = & g_\rho(\bx_n \mid \bx_n)+ \nabla g_\rho(\bx_n \mid \bx_n)^t(\bx-\bx_n)\\
&  & + \frac{1}{2}(\bx -\bx_n)^t (\bH_n + \rho \bD^t\bD)(\bx-\bx_n) 
\end{eqnarray*}
of $g_\rho(\bx \mid \bx_n)$. We now seek the best linear perturbation $\bx_n+\bG_n\bbeta$ of $\bx_n$ by minimizing $q_\rho(\bx_n+\bG_n\bbeta \mid \bx_n)$ with respect to the coefficient vector $\bbeta$. To achieve this end, we solve the stationary equation
\begin{eqnarray*}
{\bf 0} & = & \bG_n^t\nabla g_\rho(\bx_n \mid \bx_n)+\bG_n^t (\bH_n + \rho \bD^t\bD) \bG_n\bbeta
\end{eqnarray*}
and find $\bbeta=-[\bG_n^t(\bH_n + \rho \bD^t\bD) \bG_n]^{-1}\bG_n^t\nabla g_\rho(\bx_n \mid \bx_n)$, where the gradient is
\begin{eqnarray*}
	\nabla g_\rho(\bx_n \mid \bx_n) & = & \nabla h_\rho(\bx_n) \amp = \amp \nabla f(\bx_n)+\rho \bD^t[\bD\bx_n - \mathcal{P}(\bD\bx_n)].
\end{eqnarray*}
The indicated matrix inverse is just $k \times k$.

\subsection{ADMM}
\label{sec:admm}

ADMM (alternating direction method of multipliers) is a natural competitor to the proximal distance algorithms just described \citep{hong2016convergence}. ADMM is designed to minimize functions of the form $f(\bx)+g(\bD\bx)$ subject to $\bx \in C$, where $C$ is closed and convex. Splitting variables leads to the revised objective $f(\bx)+g(\by)$ subject to $\bx \in C$ and $\by=\bD\bx$. ADMM invokes the augmented Lagrangian
\begin{eqnarray*}
\mathcal{A}_\mu(\bx,\by, \blambda) & = & f(\bx)+ g(\by) + \blambda^t(\bD\bx-\by)+\frac{\mu}{2} \|\bD\bx -\by\|^2 
\end{eqnarray*}
with Lagrange multiplier $\blambda$ and step length $\mu>0$. At iteration $n+1$ of ADMM one calculates successively 
\begin{eqnarray}
\bx_{n+1} & = & \argmin_{\bx \in C} \Big[f(\bx)+
\frac{\mu}{2}\|\bD\bx-\by_{n} +\blambda_{n} \|^2\Big] \label{admm1} \\
\by_{n+1} & = & \argmin_{\by} \Big[g(\by) +
\frac{\mu}{2} \|\bD\bx_{n+1}-\by +\blambda_{n} \|^2\Big] \label{admm2} \\
\blambda_{n+1} & = & \blambda_{n}+\mu(\bD\bx_{n+1}-\by_{n+1}). \label{admm3}
\end{eqnarray}
Update (\ref{admm1}) succumbs to Newton's
method when $f(\bx)$ is smooth and $C=\mathbb{R}^p$, and update (\ref{admm2})
succumbs to a proximal map of $g(\by)$.
Update (\ref{admm3}) of the Lagrange multiplier $\blambda$ amounts to steepest ascent on the dual function. A standard extension to the scheme in (\ref{admm1}) through (\ref{admm3}) is to vary the step length $\mu$ by considering the magnitude of residuals \citep{boyd2011distributed}.
For example, letting $\br_{n} = \bD \bx - \by$ and $\bs_{n} = \mu \bD^{t} (\by_{n-1} - \by_{n})$ denote primal and dual residuals at iteration $n$, we make use of the heuristic
\begin{eqnarray*}
    \mu_{n+1} & = & \begin{cases}
        2~\mu_{n}, & \text{if}~\|\br_{n}\| / \|\bs_{n}\| > 10 \\
        \mu_{n} / 2, & \text{if}~\|\br_{n}\| / \|\bs_{n}\| < 10 \\
        \mu_{n}, & \text{otherwise}
    \end{cases}
\end{eqnarray*}
which (a) keeps the primal and dual residuals within an order of magnitude of each other, (b) makes ADMM less sensitive to the choice of step length, and (c) improves convergence.

Our problem conforms to the ADMM paradigm when $S$ is equal to the Cartesian product $\prod_{i=1}^r S_i$ and 
$g(\by) = \frac{\rho}{2} \dist(\by,S)^2$. Fortunately, the $\by$ update (\ref{admm2}) reduces to a simple formula \citep{bauschke2017convex}. To derive this formula, note that the proximal map $\by=\prox_{\alpha g}(\bz)$ satisfies the stationary condition
\begin{eqnarray*}
{\bf 0} & = & \by-\bz+\alpha [\by-\mathcal{P}(\by)] 
\end{eqnarray*}
for any $\bz$, including $\bz = \bD\bx_{n+1}+\blambda_{n}$, and any $\alpha$, including $\alpha = \rho/\mu$. Since the projection map $\mathcal{P}(\by)$ has the constant value $\mathcal{P}(\bz)$ on the line segment $[\bz,\mathcal{P}(\bz)]$, the value 
\begin{eqnarray*}
\prox_{\alpha g}(\bz) & = & \frac{\alpha}{1+\alpha}\mathcal{P}(\bz)+\frac{1}{1+\alpha}\bz
\end{eqnarray*}
satisfies the stationary condition. Because the explicit update  (\ref{admm2}) for $\by$ decreases the Lagrangian even when $S$ is nonconvex, we will employ it generally.   

The $\bx$ update (\ref{admm1}) is given by the proximal map $\prox_{\mu^{-1}f}(\blambda_n-\by_n)$ when $S=\mathbb{R}^p$ and $\bD=\bI$. Otherwise, the  update of  $\bx$ is more problematic. Assuming $f(\bx)$ is smooth and $S=\mathbb{R}^p$,  Newton's method gives the approximate update
\begin{eqnarray*}
\bx_{n+1} & = & \bx_n - \Big[d^2f(\bx_n)+ \mu \bD^t\bD \Big]^{-1} \Big[\nabla f(\bx_n)+\mu \bD^t (\bD\bx_n-\by_{n}+\blambda_{n})\Big]. 
\end{eqnarray*}
Our earlier suggestion of replacing $d^2f(\bx_n)$ by a positive definite approximation also applies here. Let us emphasize that ADMM  eliminates the need for distance majorization. Although distance majorization is convenient, it is not necessarily a tight majorization. Thus, one can hope to see gains in rates of convergence. Balanced against this positive is the fact that ADMM is often slow to converge to high accuracy.

\subsection{Proximal Distance Iteration}

We conclude this section by describing proximal distance algorithms in pseudocode. As our theoretical results will soon illustrate, the choice of penalty parameter $\rho$ is tied to the convergence rate of any proximal distance algorithm. Unfortunately, a large value for $\rho$ is necessary for the iterates to converge to the constraint set $S$.
We ameliorate this issue by slowly sending $\rho \to \infty$ according to annealing schedules from the family of geometric progressions $\rho(t) =r^{t-1}$ with $t \ge 1$. Here we parameterize the family by an initial value $\rho = 1$ and a multiplier $r>1$. Thus, our methods approximate solutions to $\min f(\bx)$ subject to $\bD \bx \in S$ by solving a sequence of increasingly penalized subproblems, $\min_{\bx} h_{\rho(t)}(\bx)$.
In practice we can only solve a finite number of subproblems so we prescribe the following convergence criteria
\begin{eqnarray}
    \label{cond:h}
    \|\nabla h_{\rho(t)}(\bx_{n})\| & \le & \delta_{h}, \\
    \label{cond:d}
    \dist(\bD\bx_{t}, S) & \le & \delta_{d}, \text{ or } \\
    \label{cond:q}
    |\dist(\bD\bx_{t}, S) - \dist(\bD\bx_{t-1}, S)| & \le & \delta_{q} [\dist(\bD\bx_{t-1}, S) + 1]
\end{eqnarray}
Condition \eqref{cond:h} is a guarantee that a solution estimate $\bx_{n}$ is close to a stationary point after $n$ inner iterations for the fixed value of $\rho =\rho(t)$. In conditions \eqref{cond:d} and 
\eqref{cond:q}, the vector $\bx_{t}$ denotes the $\delta_{h}$-optimal solution estimate once condition \eqref{cond:h} is satisfied for a particular subproblem along the annealing path.
Condition \eqref{cond:d} is a guarantee that solutions along the annealing path adhere to the fusion constraints at level $\delta_{d}$.
In general, condition \eqref{cond:d} can only be satisfied for large values of $\rho(t)$. Finally, condition \eqref{cond:q} is used to terminate the annealing process if the relative progress made in decreasing the distance penalty becomes too small as measured by $\delta_{q}$. Algorithm \ref{alg:pdi} summarizes the flow of proximal distance iteration, which uses Nesterov acceleration in inner iterations. Warm starts in solving subsequent subproblem are implicit in our formulation.
\begin{algorithm}[tbp]
\caption{Proximal Distance Iteration}
\label{alg:pdi}
\begin{algorithmic}[1]
    \footnotesize
    \State Set tolerances $\delta_{h}, \delta_{d}, \delta_{q}$, and annealing schedule $\{\rho(t)\}_{t \ge 1}$.
    \State Set maximum outer iterations $i_{\mathrm{outer}}$ and maximum inner iterations $i_{\mathrm{inner}}$.
    \State Set threshold for Nesterov acceleration, $i_{\mathrm{Nesterov}}$.
    \State Initialize $\bx$ and $q_{0} \gets \dist(\bD\bx, S)$.
    \State Initialize counter for Nesterov acceleration $i \gets 1$.
    \For {outer iterations $t = 1,2,\ldots, i_{\mathrm{outer}}$}
        \State Initialize $\bx_{1} \gets \bx$ and $\bz_{1} \gets \bx$.
        \For {inner iterations $n = 1,2,\ldots, i_{\mathrm{inner}}$}
            \If {$\|\nabla h_{\rho(t)}(\bx_{n})\| \le \delta_{h}$}
                \State Break inner loop.
            \Else
                \State Solve the subproblem $\bx_{n+1} \gets \argmin g_{\rho(t)}(\bx \mid \bz_{n})$.
                \Comment{e.g. using Newton, SD, or ADMM}
                \If {$h_{\rho(t)}(\bx_{n+1}) < h_{\rho(t)}(\bx_{n})$ AND $n \ge i_{\mathrm{Nesterov}}$}
                    \Comment{stability check before accelerating}
                    \State Accelerate $\bz_{n+1}\gets \bx_{n+1} + \frac{i-1}{i+2} (\bx_{n+1} - \bx_{n})$.
                    \State Increment $i \gets i+1$.
                \Else
                    \State Reset Nesterov acceleration; $i \gets 1$ and $\bz_{n+1} \gets \bx_{n+1}$.
                \EndIf
            \EndIf
        \EndFor

        \State Update $\bx \gets \bx_{n}$ and set $q_{t} \gets \dist(\bD\bx, S)$.
        \If {$q_{t} < \delta_{d}$ OR $|q_{t} - q_{t-1}| < \delta_{q}[1 + q_{t-1}]$}
            \State Break outer loop.
        \EndIf
    \EndFor
\end{algorithmic}
\end{algorithm}

\section{Convergence Analysis: Convex Case}
\label{sec:analysis}

Let us begin by establishing the existence of a minimum point. Further constraints on $\bx$ beyond those imposed in the distance penalties are ignored or rolled into the essential domain of $f(\bx)$ when $f(\bx)$ is convex. As noted earlier, we can assume a single fusion matrix $\bD$ and a single closed convex constraint set $S$. In such setting we have the following result. Proofs are deferred to Section \ref{sec:proofs}.
\begin{secprop} \label{propositiona}
Suppose the convex function $f(\bx)$ on $\mathbb{R}^p$ possesses a unique minimum point $\by$ on the closed convex set $T = \bD^{-1}(S)$. Then for all sufficiently large $\rho$, the objective 
$h_\rho(\bx) = f(\bx)+\frac{\rho}{2} \dist(\bD\bx,S)^2$ is coercive and therefore attains its minimum value.
\end{secprop}

Next we show that the majorization surrogate defined in (\ref{surrogate_fun}) attains its minimum value for large enough $\rho$.

\begin{secprop} \label{propositionb}
Under the conditions of Proposition \ref{propositiona}, for sufficiently large $\rho$, every surrogate $g_\rho(\bx \mid \bx_n) = f(\bx)+\frac{\rho}{2}\|\bD\bx-\mathcal{P}(\bD\bx_n)\|^2$ is coercive and therefore attains its minimum value. If  
\begin{eqnarray*}
f(\bx) & \ge &  f(\bx_n)+\bv_n^t(\bx-\bx_n)+\frac{1}{2}
(\bx-\bx_n)^t\bA(\bx-\bx_n)
\end{eqnarray*}
for all $\bx$ and some positive semidefinite matrix $\bA$ and subgradient $\bv_n$ at $\bx_n$, and if the inequality $\bu^t \bA \bu>0$ holds whenever $\|\bD\bu\|=0$ and $\bu \ne {\bf 0}$, then for $\rho$ sufficiently large, $g_\rho(\bx \mid \bx_n)$ is strongly convex and hence coercive.
\end{secprop}

We continue to assume that $f(\bx)$ and $S$ are convex and that a minimum point 
\begin{eqnarray}
\bx_{n+1} & \in &  \underset{\bx \in S}\argmin \; g_\rho(\bx \mid \bx_n) \label{MM_iterates}
\end{eqnarray}
of the surrogate $g_\rho(\bx \mid \bx_n)$ is available. Uniqueness of $\bx_{n+1}$ holds when $g_\rho(\bx \mid \bx_n)$ is strictly convex. The constraint set $S$ is implicitly captured by the essential domain of $f(\bx)$. Our earlier research shows that moving some constraints to the essential domain of $f(\bx)$ is sometimes helpful \citep{keys2019proximal,lange2016mm}. In any event, in our ideal convex setting we have a first convergence result for fixed $\rho$.	
\begin{secprop}
\label{propositionc}
Supposes a) that $S$ is closed and convex, b) that the loss $f(\bx)$ is convex and differentiable, and
c) that the constrained problem possesses a unique minimum point. For $\rho$ sufficiently large, let $\bz_\rho$ denote a minimal point of the objective $h_\rho(\bx)$ defined by equation (\ref{obj_fun}). Then the MM iterates (\ref{MM_iterates}) satisfy  
\begin{eqnarray*}
0 \le & h_{\rho}(\bx_n) - h_\rho(\bz_\rho) \amp \le \amp
\frac{\rho}{2(n+1)}\|\bD(\bz_\rho-\bx_0) \|^2.
\end{eqnarray*}
Furthermore, the iterate values $h_\rho(\bx_n)$ systematically decrease.
\end{secprop}

In even more restricted circumstances, one can prove linear convergence of function values in the framework of \citep{karimi2016linear}. 
\begin{secprop} \label{proposition0}
Suppose that $S$ is a closed and convex set and that the loss $f(\bx)$ is $L$-smooth and $\mu$-strongly convex. Then the objective $h_\rho(\bx) = f(\bx) +\frac{\rho}{2}\dist(\bD \bx,S)^2$ possesses a unique minimum point $\bz_\rho$, and the proximal distance iterates $\bx_n$ satisfy
\begin{eqnarray*}
0 & \! \le \! & h_\rho(\bx_{n})-h_\rho(\bz_\rho) \amp \!\le \! \amp \Big[1-\frac{\mu^2}
{2(L+\rho\|\bD\|^2)^2}\Big]^n[h_\rho(\bx_{0})-h_\rho(\bz_\rho)].
\end{eqnarray*}
\end{secprop}

Convergence of ADMM is well studied in the optimization literature \citep{beck2017first}.  Appendix \ref{appendix:ADMM} summarizes the main findings. 

\section{Convergence Analysis: General Case}
\label{sec:nonconvex}

We now depart the comfortable confines of convexity. Let us first review the notion of a Fr\'echet subdifferential \citep{kruger2003frechet}. If $h(\bx)$ is a function mapping $\mathbb{R}^p$ into $\mathbb{R} \cup \{+\infty\}$, then its Fr\'echet subdifferential at $\bx \in \dom f$ is defined as
\begin{eqnarray*}
\partial^F h(\bx) & = & \Big\{\bv: \liminf_{\by \to \bx}
\frac{h(\by)-h(\bx) - \bv^t(\by-\bx)}{\|\by-\bx\|} \ge 0 \Big\}.
\end{eqnarray*}
The set $\partial^F h(\bx)$ is closed, convex, and possibly empty. If $h(\bx)$ is convex, then $\partial^F h(\bx)$ reduces to its convex subdifferential. If $h(\bx)$ is differentiable, then $\partial^F h(\bx)$ reduces to its ordinary differential. At a local minimum $\bx$, Fermat's rule ${\bf 0} \in \partial^F h(\bx)$ holds. For a locally Lipschitz and directionally differentiable function, the Fr\'echet subdifferential becomes
\begin{eqnarray*}
\partial^F h(\bx) & = & \Big\{\bv: d_{\bu}h(\bx) \ge \bv^t\bu \; \text{for all directions} \; \bu \Big\}.
\end{eqnarray*}
Here $d_{\bu}h(\bx)$ is the directional derivative of $h(\bx)$ at $\bx$ in the direction $\bu$. This result makes it clear that at a critical point, all directional derivatives are flat or point uphill.

We will also need some notions from algebraic geometry \citep{bochnak2013real}. 
For simplicity we focus on the class of semialgebraic functions and the corresponding class of semialgebraic subsets of $\mathbb{R}^p$. The latter is the smallest class that:
\begin{description}
\item[(a)] contains all sets of the form
$\{\bx: q(\bx)>0 \}$ for a polynomial $q(\bx)$ in $p$ variables,
\item[(b)] is closed under the formation of finite unions, finite intersections, set complementation, and Cartesian products. 
\end{description}
A function $a:\mathbb{R}^p \mapsto \mathbb{R}^r$ is said to be semialgebraic if its graph is a semialgebraic set of $\mathbb{R}^{p+r}$. The class of real-valued semialgebraic functions contains all polynomials $p(\bx)$ and all $0$/$1$ indicators of algebraic sets. It is closed under the formation of sums and products and therefore constitutes a commutative ring with identity. The class is also closed under the formation of absolute values, reciprocals when $a(\bx) \ne 0$, $n$th roots when $a(\bx) \ge 0$, and maxima $\max\{a(\bx),b(\bx)\}$ and minima $\min\{a(\bx),b(\bx)\}$. Finally, the composition of two semialgebraic functions is semialgebraic.

For our purposes it is crucial that the Euclidean distance $\dist(\bx, S)$ to a semialgebraic set $S$ is a semialgebraic function. In view of the closure properties of such functions, the function $\frac{\rho}{2}\dist(\bD\bx, S)^2$ is also  semialgebraic. Sets such as the nonnegative orthant $\mathbb{R}^p_+$ and the unit sphere $S^{p-1}$ are semialgebraic. The next proposition buttresses several of our numerical examples.
\begin{secprop} \label{order_stats_prop}
The order statistics of a finite set $\{f_i(\bx)\}_{i=1}^n$ of semialgebraic functions are semialgebraic. Hence, sparsity sets are semialgebraic.
\end{secprop}

The next proposition is an elaboration and expansion of known results \citep{attouch2010proximal,bolte2007lojasiewicz,cui2018composite,kang2015global,le2018convergence} and was featured in our previous paper \citep{keys2019proximal}. 

\begin{secprop} 
\label{MMconvergence}
In an MM algorithm suppose the objective $h(\bx)$ is coercive, continuous, and subanalytic and all surrogates $g(\bx \mid \bx_n)$ are continuous, $\mu$-strongly convex, and satisfy the Lipschitz condition
\begin{eqnarray*}
\|\nabla g(\ba \mid \bx_n)-\nabla g(\bb \mid \bx_n) \|
& \le & L \|\ba-\bb\| 
\end{eqnarray*}
on the compact set $\{\bx: h(\bx) \le h(\bx_0)\}$. Then the MM iterates 
$\bx_{n+1} = \argmin_{\bx} g(\bx \mid \bx_n)$ converge to a stationary point.  
\end{secprop} 

Proposition \ref{MMconvergence} applies to proximal distance algorithms under the right hypotheses. Note that semialgebraic sets and functions are automatically subanalytic.
Before stating a precise result, let us clarify the nature of the Fr\'echet subdifferential in the current setting. This entity is determined by the identity
\begin{eqnarray*}
\frac{1}{2}\dist(\bD\bx,S)^2 & = & \frac{1}{2}\min_{\bz \in S}\|\bD\bx-\bz\|^2,
\end{eqnarray*}
Danskin's theorem yields the directional derivative
\begin{eqnarray*}
d_{\bv}\frac{1}{2}\dist(\bD\bx,S)^2 & = &\min_{\bz\in S(\bx)} (\bD\bx-\bz)^t\bD\bv,
\end{eqnarray*}
where $S(\bx)$ is the solution set where the minimum is attained. The Frechet differential 
\begin{eqnarray*}
\partial^F h_\rho(\bx) & = & \nabla f(\bx)+\rho \{\bu:(\bD\bx-\bz)^t\bD\bv\ge \bu^t\bv, \, \bz \in S(\bx) \; \text{and all} \; \bv\} \\
& = & \nabla f(\bx)+\rho \{\bu:\bu = \bD^t(\bD\bx-\bz), \,\bz \in S(\bx)\}
\end{eqnarray*}
holds owing to Corollary 1.12.2 and Proposition 1.17 of \citep{kruger2003frechet} since $\dist(\bD\bx,S)^2$ is locally Lipshitz. The latter fact follows from the identity $a^2-b^2 = (a+b)(a-b)$ with $a=\dist(\bD\by,S)^2$ and $b=\dist(\bD\bx,S)^2$, given that $\dist(\bw,S)$ is Lipschitz and bounded on bounded sets. 

In any event, a stationary point $\bx$ satisfies $\bzero = \nabla f(\bx)+\rho\bD^t(\bD\bx-\bz)$ for all $\bz \in S(\bx)$. As we expect, the stationary condition is necessary for $\bx$ to furnish a global minimum. Indeed, if it fails, take $\bz \in S(\bx)$ with surrogate satisfying $\nabla g_\rho(\bx \mid \bx) \ne \bzero$. Then the negative gradient $-\nabla g_\rho(\bx \mid \bx)$ is a descent direction for $g_\rho(\bx \mid \bx)$, which majorizes $h_\rho(\bx)$. Hence, $-\nabla g_\rho(\bx \mid \bx)$ is also a descent direction for $h_\rho(\bx)$. This conclusion is inconsistent with $\bx$ being a local minimum of the objective. 
\medskip

The next proposition proves convergence for a wide class of fused models.
\begin{secprop} 
\label{non-convex-convergence} Suppose in our proximal distance setting that $\rho$ is sufficiently large, the closed constraint sets $S_i$ and the loss $f(\bx)$ are semialgebraic, and $f(\bx)$ is differentiable with a locally Lipschitz gradient. Under the coercive assumption made in Proposition \ref{propositionb}, the proximal distance iterates $\bx_n$ converge to a stationary point of the objective $h_\rho(\bx)$.
\end{secprop} 

For sparsity constrained problems, one can establish a linear rate of convergence under the right hypotheses.
\begin{secprop} 
\label{sparsity-convergence} Suppose in our proximal distance setting that $\rho$ is sufficiently large, the constraint set is a sparsity
set $S_k$, and the loss $f(\bx)$ is semialgebraic, strongly convex, and possesses a Lipschitz gradient. Then the proximal distance iterates $\bx_n$ converge linearly to a stationary point $\bx_\infty$  provided $\bD\bx_\infty$ has $k$ unambiguous largest components in magnitude. When the rows of $\bD$ are unique, the complementary set of points $\bx$ where $\bD\bx$ has $k$ ambiguous largest components in magnitude has Lebesgue measure $0$.
\end{secprop} 

\section{Numerical Examples}

This section considers five concrete examples of constrained optimization amenable to distance majorization with fusion constraints, with $\bD$ denoting the fusion matrix in each problem. In each case, the loss function is both strongly convex and differentiable.  The specific examples that we consider are the  metric projection problem, convex regression,  convex clustering, image denoising with a total variation penalty, and projection of a matrix to one with a better condition number. Each example is notable for the large number of fusion constraints and projections to convex constraint sets, except in convex clustering.
In convex clustering we encounter a sparsity constraint set.
Quadratic loss models feature prominently in our examples.
Interested readers may consult our previous work for nonconvex examples with $\bD = \bI$ \citep{keys2019proximal,xu2017generalized}.

\subsection{Mathematical Descriptions}

Here we provide the mathematical details for each example.

\subsubsection{Metric Projection}

Solutions to the metric projection  problem restore transitivity to noisy distance data for the $m$ nodes of a graph \citep{brickell_metric_2008,sra2005triangle}. The data are encoded in an $m \times m$ dissimilarity matrix $\bY = (y_{ij})$ with nonnegative weights in the  matrix $\bW = (w_{ij})$. The metric projection
problem requires finding the symmetric semi-metric $\bX = (x_{ij})$ minimizing
\begin{eqnarray*}
	f(\bX) & = &  \sum_{i>j} w_{ij} (x_{ij} - y_{ij})^{2} 
\end{eqnarray*}
subject to all $\binom{m}{2}$ nonnegativity constraints $x_{ij} \ge 0$ and all $3 \binom{m}{3}$ triangle inequality constraints $x_{ij} - x_{ik}-x_{kj} \le 0$. The diagonal entries of $\bY$, $\bW$, and $\bX$ are zero by definition. The fusion matrix $\bD$ has $\binom{m}{2}+3 \binom{m}{3}$ rows, and the projected value of $\bD \bX$ must fall in the set $S$ of symmetric matrices satisfying all pertinent constraints.

One can simplify the required projection by stacking the nonredundant entries along each successive column of $\bX$ to create a vector $\bx$ with $\binom{m}{2}$ entries. This captures the lower triangle of $\bX$. The sparse matrix $\bD$ is correspondingly redefined to be $[\binom{m}{2}+3 \binom{m}{3}] \times \binom{m}{2}$. These maneuvers simplfy constraints to $\bD \bx \ge {\bf 0}$, and projection involves sending each entry $u$ of $\bD\bx$ to $\max\{0, u\}$.
Putting everything together, the objective to minimize is
\begin{eqnarray*}
    h_{\rho}(\bx)
    & = &
    \frac{1}{2} \|\bW^{1/2}(\bx - \by)\|^{2}_{2}
    + \frac{\rho}{2} \dist(\bT \bx, \mathbb{R}_{+}^{p_{1}})^2
    + \frac{\rho}{2} \dist(\bx, \mathbb{R}_{+}^{p_{2}})^2,
\end{eqnarray*}
where $\bD$ consists of blocks $\bT$ and $\bI_{p_2}$ and $p_{1}$ and $p_{2}$ count the number of triangle inequality and nonnegativity constraints, respectively. The linear system $(\bI + \rho \bD^{t} \bD) \bx = \bb$ appears in both the MM and ADMM updates for $\bx_{n}$. Application of the Woodbury and Sherman-Morrison formulas yield an exact solution to the linear system and allow one to forgo iterative methods.
The interested reader may consult Appendix \ref{appendix:example1} for further details.

\subsubsection{Convex Regression}

Convex regression is a nonparametric method for estimating a regression function under shape constraints. Given $m$ responses $y_i$ and corresponding predictors $\bx_i \in \mathbb{R}^d$, the goal is to find the convex function $\psi(\bx)$ minimizing the sum of squares $\frac{1}{2}\sum_{i=1}^{m} [y_{i} - \psi(\bx_{i})]^{2}$. Asymptotic and finite sample properties of this convex estimator have been described in detail by \cite{seijo_nonparametric_2011}. The convex regression program can be restated as the finite dimensional problem of finding the value $\theta_i$ and subgradient 
$\bxi_{i} \in \mathbb{R}^{d}$ of $\psi(\bx)$ at each sample point $(y_i,\bx_i)$. Convexity imposes the supporting hyperplane constraint $\theta_{j} + \bxi_{j}^t (\bx_{i} - \bx_{j}) \le \theta_{i}$ for each pair $i \ne j$. Thus, the problem becomes one of minimizing $\frac{1}{2}\|\by-\btheta\|^2$
subject to these $m(m-1)$ inequality constraints. In the proximal distance framework, we must minimize
\begin{eqnarray*}
h_\rho(\btheta,\bXi) & = & \frac{1}{2} \|\by-\btheta\|^{2}
+ \frac{\rho}{2} \dist(\bA \btheta + \bB \bXi, \mathbb{R}_{-}^{m(m-1)})^{2},
\end{eqnarray*}
where $\bD = [\bA~\bB]$ encodes the required fusion matrix. The reader may consult Appendix \ref{appendix:example2} for a description of each algorithm map.

\subsubsection{Convex Clustering}
Convex clustering of $m$ samples based on $d$ features can be formulated in terms of the regularized objective
\begin{eqnarray*}
    \label{eq:regularized-objective}
    F_{\gamma}(\bU)
    & = &
    \frac{1}{2} \sum_{i=1}^{m} \|\bu_{i} - \bx_{i}\|^{2}
    +
    \gamma \sum_{i > j} w_{ij} \|\bu_{i} - \bu_{j}\|,
\end{eqnarray*}
based on columns $\bx_{i}$ and $\bu_{i}$ of $\bX \in \mathbb{R}^{d \times m}$ and $\bU \in \mathbb{R}^{d \times m}$, respectively.
Here each $\bx_{i} \in \mathbb{R}^{d}$ is a sample feature vector and the corresponding $\bu_{i}$ represents its centroid assignment.
The predetermined weights $w_{ij}$ have a graphical interpretation under which similar samples have positive edge weights $w_{ij}$ and distant samples have $0$ edge weights. The edge weights are chosen by the user to guide the clustering process. In general, minimization of $F_{\gamma}(\bU)$ separates over the connected components of the graph. To allow all sample points to coalesce into a single cluster, we assume that the underlying graph is connected. The regularization parameter $\gamma > 0$ tunes the number of clusters in a nonlinear fashion and potentially captures hierarchical information. Previous work establishes that the solution path $\bU(\gamma)$ varies continuously with respect to $\gamma$ \citep{chi_splitting_2015}. Unfortunately, there is no explicit way to determine the number of clusters entailed by a particular value of $\gamma$.

Alternatively, we can attack the problem using sparsity and distance majorization.
Consider the penalized objective
\begin{eqnarray*}
    h_{\rho,k}(\bU)
   &  = &
    \frac{1}{2}\|\bU - \bX\|_{F}^{2}
    +
    \frac{\rho}{2} \dist(\bU \bD, S_{k})^{2}.
\end{eqnarray*}
The fusion matrix $\bD$ has $\binom{m}{2}$ columns $w_{ij}(\be_{i} - \be_{j})$ and serves to map the centroid matrix $\bU$ to a $d \times \binom{m}{2}$ matrix $\bV$ encoding the weighted differences 
$w_{ij}(\bu_{i} - \bu_{j})$. The members of the sparsity set $S_{k}$ are $d \times \binom{m}{2}$ matrices with at most $k$ non-zero columns. Projection of $\bU \bD$ onto the closed set $S_{k}$ forces some centroid assignments to coalesce, and is straightforward to implement by sorting the Euclidean lengths of the columns of $\bU \bD$ and sending to $\bf 0$ all but the $k$ most dominant columns. Ties are broken arbitrarily.

Our sparsity-based method trades the continuous penalty parameter $\gamma > 0$ in the previous formulation for an integer sparsity index $k \in \{0,1,2,\ldots,\binom{m}{2}\}$. For example with $k = 0$, all differences $\bu_{i} - \bu_{j}$ are coerced to $\bf 0$, and all sample points cluster together. The other extreme $k = \binom{m}{2}$ assigns each point to its own cluster. The size of the matrices $\bD$ and $\bU \bD$ can be reduced by
discarding column pairs corresponding to $0$ weights.
Appendix \ref{appendix:example3} describes the projection onto sparsity sets and provides further details.

\subsubsection{Total Variation Image Denoising}

To approximate an image $\bU$ from a noisy input $\bW$ matrix, \cite{rudin_nonlinear_1992}
regularize a loss function $f(\bU)$ by a total variation (TV) penalty. After discretizing the problem, the least squares loss leads to the objective 
\begin{eqnarray*}
F_\gamma(\bU) & = &	\sum_{i,j} (U_{i,j} - W_{i,j})^{2}+
\gamma \sum_{i,j} \sqrt{(U_{i+1,j} - U_{i,j})^{2} + (U_{i,j+1} - U_{i,j})^{2}},
\end{eqnarray*}
where $\bU, \bW \in \mathbb{R}^{m \times p}$ are rectangular monochromatic images and $\gamma$ controls the strength of regularization. The anisotropic norm
\begin{eqnarray*}
\mathrm{TV}_{1}(\bU)
& = & \sum_{i,j} |U_{i+1,j} - U_{i,j}| + |U_{i,j+1} - U_{i,j}|
\amp = \amp	\|\bD_{m} \bU\|_{1} + \|\bU \bD_{p}^{t}\|_{1}
\end{eqnarray*}
is often preferred because it induces sparsity in the differences. Here $\bD_{p}$ is the forward difference operator on $p$ data points. Stacking the columns of $\bU$ into a vector $\bu = \vec(\bU)$
allows one to identify a fusion matrix $\bD$ and write $\mathrm{TV}_{1}(\bU)$ compactly as $\mathrm{TV}_{1}(\bu) = \|\bD \bu\|_{1}$. In this context we reformulate the denoising problem
as minimizing $f(\bu)$ subject to the set constraint $\|\bD \bu\|_{1} \le \gamma$. This revised formulation directly quantifies the quality of a solution in terms of its total variation and brings into play fast pivot-based algorithms for projecting onto multiples of the $\ell_{1}$ unit ball \citep{condat2016}.
Appendix \ref{appendix:example4} provides descriptions of each algorithm.

\subsubsection{Projection of a Matrix to a Good Condition Number}

Consider an $m \times p$ matrix $\bM$ with $m \ge p$ and full singular value decomposition $\bM=\bU\bSigma\bV^t$. The condition number of $\bM$ is the ratio $\sigma_{\rm \max}/\sigma_{\rm min}$ of the largest to the smallest singular value of $\bM$. We denote the diagonal of $\bSigma$ as $\bsigma$. Owing to the von Neumann-Fan inequality, the closest matrix $\bN$ to $\bM$ in the Frobenius norm has the singular value decomposition $\bN=\bU\bX\bV^t$, where the diagonal $\bx$ of $\bX$ satisfies inequalities pertinent to a decent condition number \citep{borwein2010convex}. Suppose $c \ge 1$ is the maximum condition number. Then every pair $(x_i,x_j)$ satisfies $x_i - cx_j \le 0$. Note that $x_i - cx_i > 0$ if and only if $x_i <0$. Thus, nonnegativity of the entries of $\bx$ is enforced. The proximal distance approach to the condition number projection problem invokes the objective and majorization
\begin{eqnarray*}
h_{\rho}(\bx) & = & \frac{1}{2}\|\bsigma-\bx\|^2 + \frac{\rho}{2}\sum_{(i,j)} \dist(x_i - c x_j,\mathbb{R}_-)^2\\
& = & \frac{1}{2}\|\bsigma-\bx\|^2 + \frac{\rho}{2}\sum_{(i,j)} (x_i - c x_j)_+^2 \\
& \le & \frac{1}{2}\|\bsigma-\bx\|^2  
+ \frac{\rho}{2}\sum_{(i,j)} (x_i - c x_j - q_{nij})^2
\end{eqnarray*}
at iteration $n$, where $q_{nij}= \min\{x_{ni} - c x_{nj},0\}$. We can write the majorization more concisely as 
\begin{eqnarray*}
h_{\rho}(\bx) & \le & \frac{1}{2}\|\bA_\rho\bx -\br_n\|^2, \quad 
\bA_\rho \amp = \amp \begin{pmatrix} \bI_p \\ \sqrt{\rho}\bD \end{pmatrix}, \quad
\br_n \amp = \amp \begin{pmatrix} \bsigma \\ \sqrt{\rho}\vec{\bQ_n} \end{pmatrix},
\end{eqnarray*}
where $\vec{\bQ_n}$ stacks the columns of $\bQ_n=(q_{nij})$ and the $p^2 \times p$ fusion matrix $\bD$ satisfies $(\bD \bx)_{k} = x_i - c x_j$ for each component $k$. The minimum of the surrogate occurs at the point $\bx_{n+1} = (\bA_\rho^t\bA_\rho)^{-1}\bA_\rho^t\br_n$.
This linear system can be solved exactly. Appendix \ref{appendix:example5} provides additional details.

\subsection{Numerical Results}

Our numerical experiments compare various strategies for implementing Algorithm \ref{alg:pdi}.
We consider two variants of proximal distance algorithms.
The first directly minimizes the majorizing surrogate (MM), while the second performs steepest descent (SD) on it.
In addition to the aforementioned methods, we tried the subspace MM algorithm described in Section \ref{sec:sd}.
Unfortunately, this method was outperformed in both time and accuracy comparisons by Nesterov accelerated MM; the MM subspace results are therefore omitted.
We also compare our proximal distance approach to ADMM as described in Section \ref{sec:admm}.
In many cases updates require solving a large linear system; we found that the method of conjugate gradients sacrificed little accuracy and largely outperformed LSQR and therefore omit comparisons.
The clustering and denoising examples are exceptional in that the associated matrices $\bD^{t} \bD$ are sufficiently ill-conditioned to cause failures in conjugate gradients.
Table \ref{tab:params} summarizes choices in control parameters across each example.
\begin{table}[tbp]
    \centering
	\begin{footnotesize}
    \begin{tabular}{ccccc}
        \toprule
        & $\delta_{h}$ & $\delta_{d}$ & $\delta_{q}$ & $\rho(t) = r^{t}$ \\
        \midrule
        Metric Projection & $10^{-3}$ & $10^{-2}$ & $10^{-6}$ & $\min\{10^{8}, 1.2^{t-1}\}$ \\
        Convex Regression & $10^{-3}$ & $10^{-2}$ & $10^{-6}$ & $\min\{10^{8}, 1.2^{t-1}\}$ \\
        Convex Clustering & $10^{-2}$ & $10^{-5}$ & $10^{-6}$ & $\min\{10^{8}, 1.2^{t-1}\}$ \\
        Image Denoising & $10^{-1}$ & $10^{-1}$ & $10^{-6}$ & $\min\{10^{8}, 1.5^{t-1}\}$ \\
        Condition Numbers & $10^{-3}$ & $10^{-2}$ & $10^{-6}$ & $\min\{10^{8}, 1.2^{t-1}\}$ \\
        \bottomrule
    \end{tabular}
    \end{footnotesize}
    \caption{\label{tab:params} Summary of control parameters used in each example.}
\end{table}

We now explain example by example the implementation details behind our efforts to benchmark the three strategies (MM, SD, and ADMM) in implementing Algorithm \ref{alg:pdi}. In each case we initialize the algorithm with the solution of the corresponding unconstrained problem. Performance is assessed in terms of speed in seconds or milliseconds, number of iterations until convergence, the converged value of the loss $f(\bx)$, and the converged distance to the constraint set $\dist(\bD \bx, S)$, as described in Algorithm \ref{alg:pdi}. Additional metrics are highlighted where applicable.
The term \textit{inner iterations} refers to the number of iterations to solve a penalized subproblem $\argmin h_{\rho}(\bx)$ for a given $\rho$ whereas \textit{outer iterations} count the total number of subproblems solved.
Lastly, we remind readers that the approximate solution to $\argmin h_{\rho(t)}(\bx)$ is used as a warm start in solving $\argmin h_{\rho(t+1)}(\bx)$.

\subsubsection{Metric Projection.}

In our comparisons, we use input matrices  $\bY \in \mathbb{R}^{m \times m}$ whose iid entries $y_{ij}$ are drawn uniformly from the interval $[0,10]$ and set weights $w_{ij} = 1$. Each algorithm is allotted a maximum of $200$ outer and $10^{5}$ inner iterations, respectively, to achieve a gradient norm of $\delta_{h} = 10^{-3}$ and distance to feasibility of $\delta_{d} = 10^{-2}$. The relative change parameter is set to $\delta_{q} = 0$ and the annealing schedule is set to $\rho(t) = \min\{10^{8}, 1.2^{t-1}\}$ for the proximal distance methods. Table \ref{tab:ex1} summarizes the performance of the three algorithms. Best values appear in boldface. All three algorithms converge to a similar solution as indicated by the final loss $\|\by - \bx\|^{2}$ and distance values. It is clear that SD matches or outperforms MM and ADMM approaches on this example. Notably, the linear system appearing in the MM update admits an exact solution (see Appendix \ref{appendix:example1-update}), yet SD has a faster time to solution with fewer iterations taken.
\begin{table}[tbp]
	\centering
	\setlength{\tabcolsep}{1.2pt}
    \begin{scriptsize}
    \begin{tabular}{rcccccccccccc}
    \toprule
     & \multicolumn{3}{c}{Time (s)} & \multicolumn{3}{c}{Loss $\times 10^{-3}$} & \multicolumn{3}{c}{Distance $\times 10^{3}$} & \multicolumn{3}{c}{Iterations}\\
    \cmidrule(lr){2-4} \cmidrule(lr){5-7} \cmidrule(lr){8-10} \cmidrule(lr){11-13}
    $m$ & MM & SD & ADMM & MM & SD & ADMM & MM & SD & ADMM & MM & SD & ADMM\\
    \midrule
    16 & $0.0429$ & $\bf{0.0338}$ & $0.107$ & $\bf{0.237}$ & $\bf{0.237}$ & $\bf{0.237}$ & $\bf{9.24}$ & $\bf{9.24}$ & $\bf{9.24}$ & $4980$ & $\bf{3920}$ & $7030$\\
     & $(0.000709)$ & $(0.000137)$ & $(0.000858)$ &  &  &  &  &  &  & $(37)$ & $(37)$ & $(37)$\\
    32 & $\bf{1.24}$ & $1.28$ & $2.6$ & $\bf{1.14}$ & $\bf{1.14}$ & $\bf{1.14}$ & $\bf{9.13}$ & $\bf{9.13}$ & $\bf{9.13}$ & $16000$ & $\bf{15400}$ & $17300$\\
     & $(0.00309)$ & $(0.00617)$ & $(0.0149)$ &  &  &  &  &  &  & $(41)$ & $(41)$ & $(41)$\\
    64 & $19.5$ & $\bf{16.7}$ & $43.9$ & $\bf{4.69}$ & $\bf{4.69}$ & $\bf{4.69}$ & $\bf{8.7}$ & $\bf{8.7}$ & $\bf{8.7}$ & $30100$ & $\bf{24200}$ & $33700$\\
     & $(0.00835)$ & $(0.0148)$ & $(0.0616)$ &  &  &  &  &  &  & $(44)$ & $(44)$ & $(44)$\\
    128 & $171$ & $\bf{150}$ & $725$ & $\bf{18.4}$ & $\bf{18.4}$ & $\bf{18.4}$ & $\bf{9.44}$ & $\bf{9.44}$ & $\bf{9.44}$ & $29900$ & $\bf{23700}$ & $51900$\\
     & $(0.613)$ & $(0.28)$ & $(0.772)$ &  &  &  &  &  &  & $(44)$ & $(44)$ & $(44)$\\
    256 & $1670$ & $\bf{1570}$ & $9110$ & $\bf{75.3}$ & $\bf{75.3}$ & $\bf{75.3}$ & $\bf{8.68}$ & $\bf{8.68}$ & $\bf{8.68}$ & $32500$ & $\bf{26700}$ & $76100$\\
     & $(0.702)$ & $(5.63)$ & $(75.4)$ &  &  &  &  &  &  & $(46)$ & $(46)$ & $(46)$\\
    \bottomrule
    \end{tabular}
    \end{scriptsize}
	\caption{
	    \label{tab:ex1} Metric projection. Times are averaged over 3 replicates with standard deviations in parentheses. Reported iteration counts reflect the total inner iterations taken with outer iterations in parentheses.
    }
\end{table}
The selected convergence metrics in Figure~\ref{fig:converge} vividly illustrate stability of solutions $\bx_{\rho} = \argmin h_{\rho}(\bx)$ along an annealing path from $\rho = 1$ to $\rho = 1.2^{40} \approx 1470$.
Specifically, solving each penalized subproblem along the sequence results in marginal increase in the loss term with appreciable decrease in the distance penalty.
Except for the first outer iteration, there is minimal decrease of the loss, distance penalty, or penalized objective within a given outer iteration even as the gradient norm vanishes.
The observed tradeoff between minimizing a loss model and minimizing a nonnegative penalty is well-known in penalized optimization literature \citep[see Proposition 7.6.1 on p.~183]{beltrami1970algorithmic,lange2016mm}.
\begin{figure}[!b]
    \centering
    \includegraphics[width=0.35\textwidth]{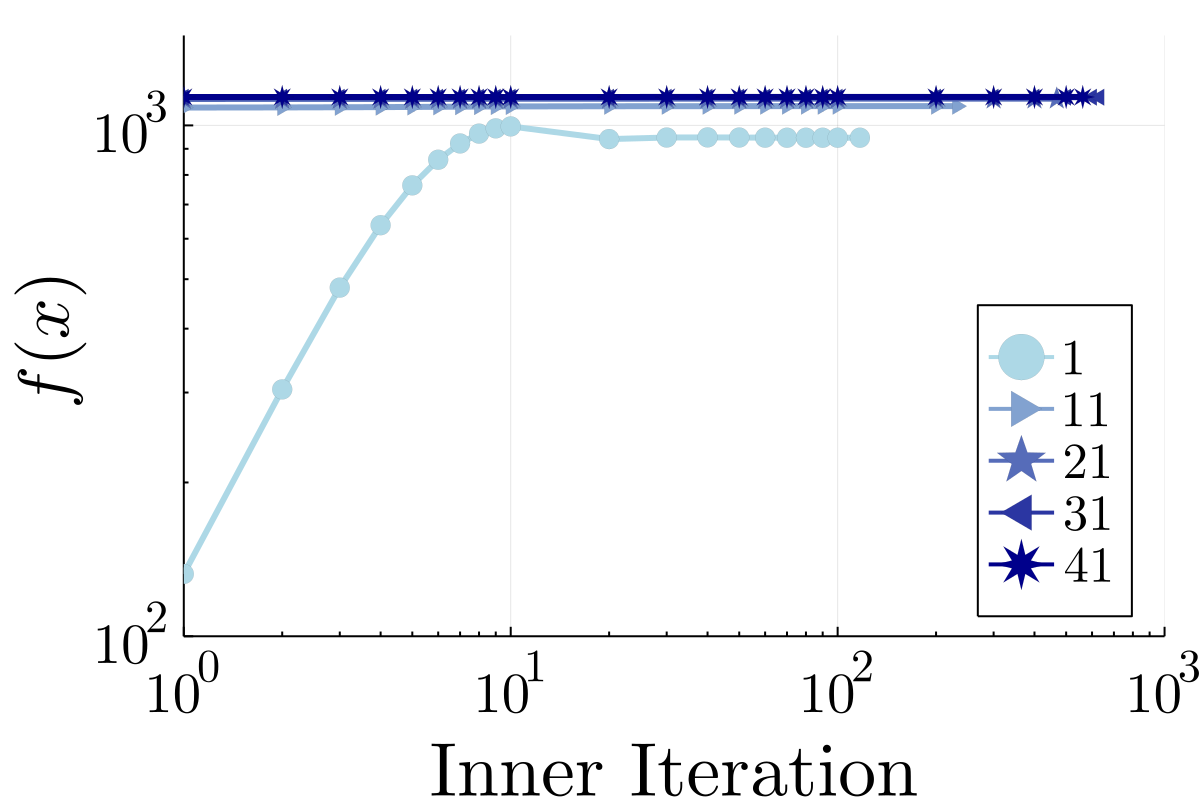}%
    \includegraphics[width=0.35\textwidth]{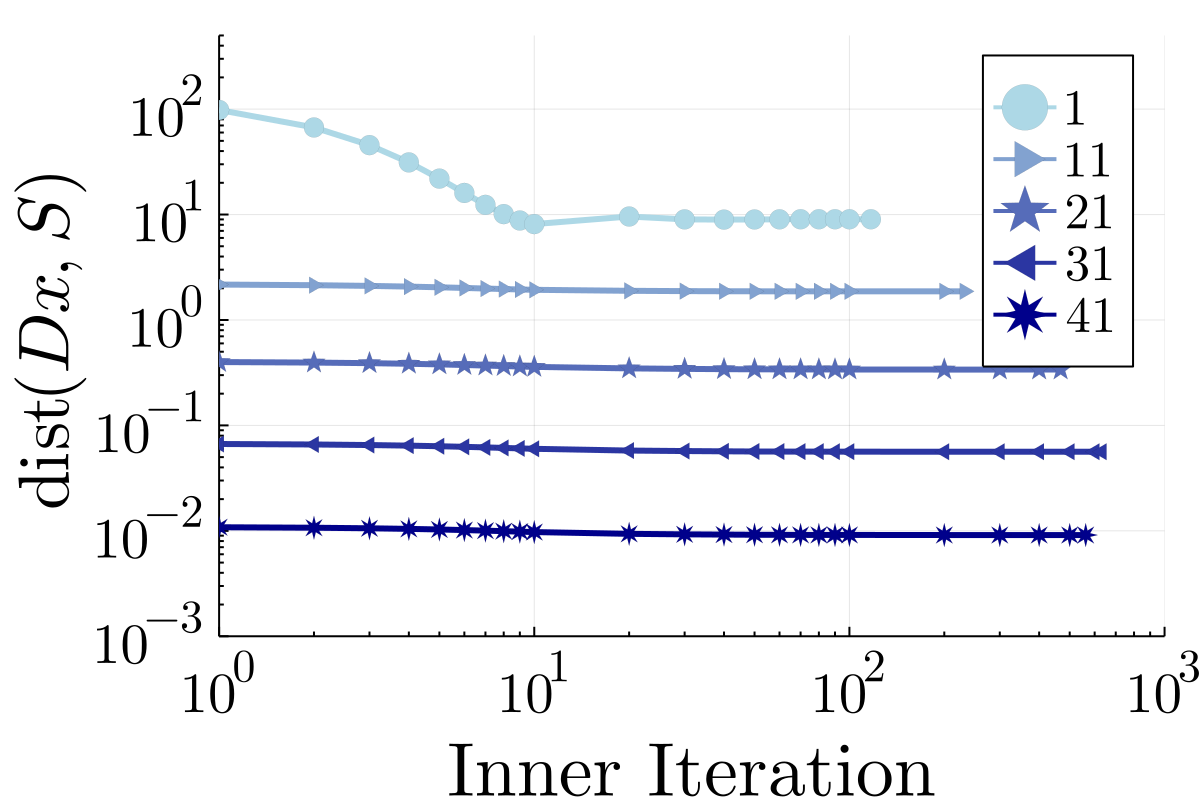}
    \includegraphics[width=0.35\textwidth]{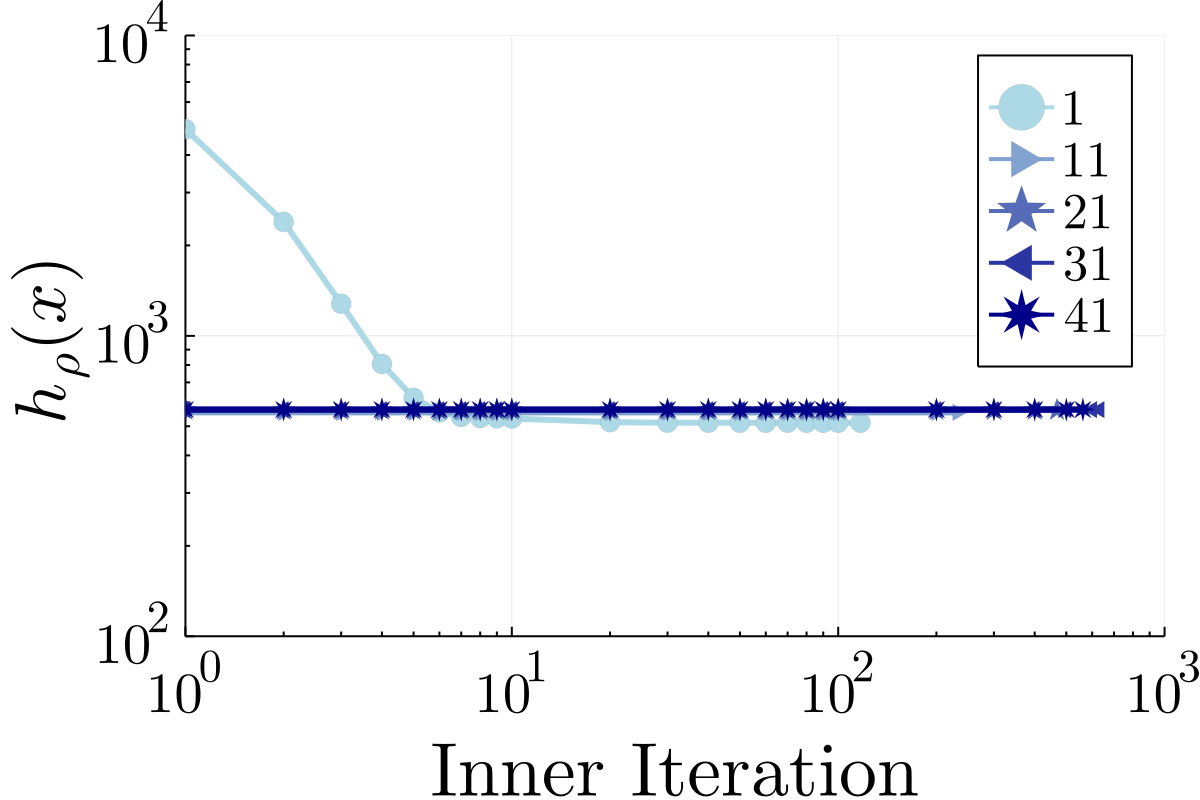}%
    \includegraphics[width=0.35\textwidth]{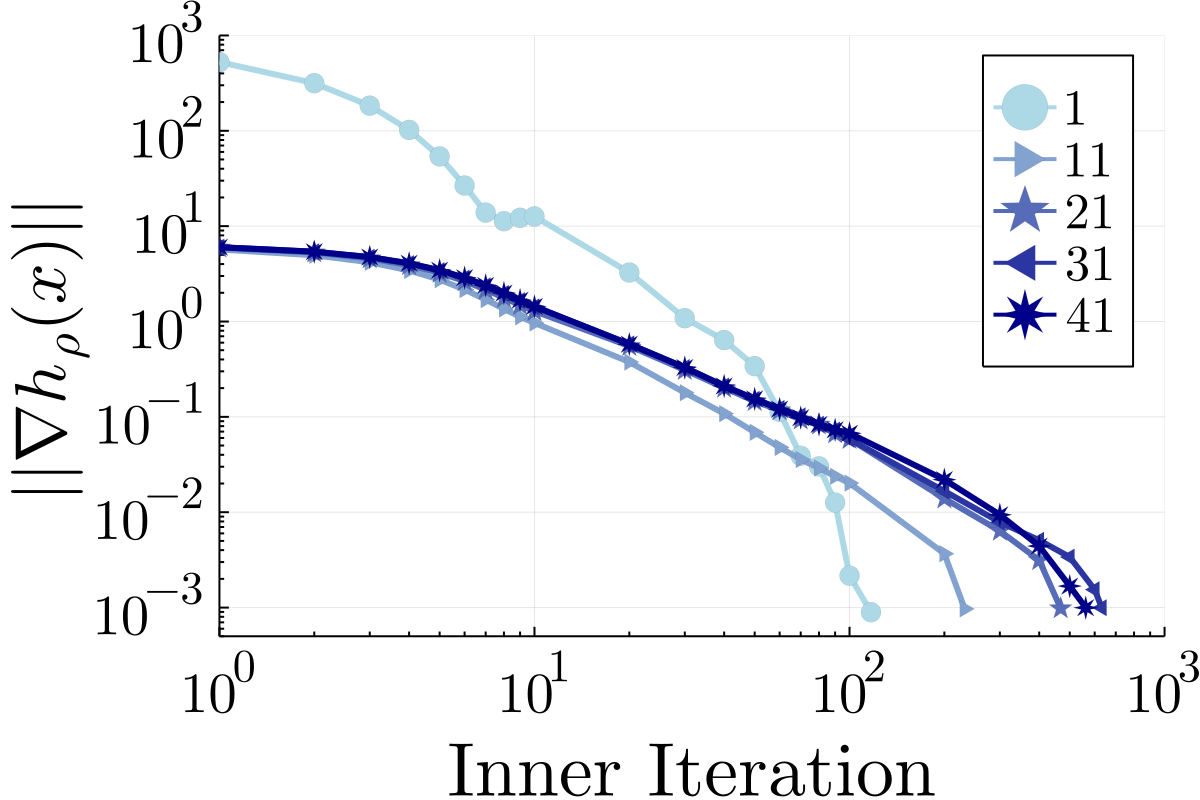}
    \caption{\label{fig:converge} Loss, distance, penalized objective, and gradient norm for SD on metric projection problem with $32$ nodes, labeled by outer iteration.}
\end{figure}
\newpage

\subsubsection{Convex Regression.}

In our numerical examples the observed functional values $y_i$ are independent Gaussian deviates with means $\psi(\bx_{i})$ and common variance $\sigma^2=0.1$. The predictors are iid deviates sampled from the uniform distribution on $[-1,1]^d$. We choose the
simple convex function $\psi(\bx_{i})=\|\bx_i\|^2$ for our benchmarks for ease in interpretation; the interested reader may consult the work of \cite{mazumder_computational_2019} for a detailed account of the applicability of the technique in general. Each algorithm is allotted a maximum of $200$ outer and $10^{4}$ inner iterations, respectively, to converge with $\delta_{h} = 10^{-3}$, $\delta_{d} = 10^{-2}$, and $\delta_{q} = 10^{-6}$.
The annealing schedule is set to $\rho(t) = \min\{10^{8}, 1.2^{t-1}\}$. 

Table \ref{tab:ex2} demonstrates that although the SD approach is appreciably faster than both MM and ADMM, the latter appear to converge on solutions with marginal improvements in minimizing the loss $\|\by - \btheta\|^{2}$, distance, and mean squared error (MSE) measured using ground truth functional values $\psi(\bx_{i})$ and estimates $\theta_{i}$. Interestingly, increasing both the number of features and samples does not necessarily increase the amount of required computational time in using a proximal distance approach; for example, see results with $d=2$ and $d=20$ features. This may be explained by sensitivity to the annealing schedule.
 \begin{table}[tbp]
  	\centering
  	\setlength{\tabcolsep}{2pt}
  	\begin{scriptsize}
  	\resizebox{\textwidth}{!}{%
    \begin{tabular}{rrcccccccccccc}
    \toprule
    &  & \multicolumn{3}{c}{Time (s)} & \multicolumn{3}{c}{Loss $\times 10^{3}$} & \multicolumn{3}{c}{Distance $\times 10^{4}$} & \multicolumn{3}{c}{MSE $\times 10^{2}$}\\
    \cmidrule(lr){3-5} \cmidrule(lr){6-8} \cmidrule(lr){9-11} \cmidrule(lr){12-14}
    $d$ & $m$ & MM & SD & ADMM & MM & SD & ADMM & MM & SD & ADMM & MM & SD & ADMM\\
    \midrule
    1 & 50 & $0.015$ & $\bf{0.0109}$ & $0.0171$ & $\bf{454}$ & $\bf{454}$ & $\bf{454}$ & $\bf{87.3}$ & $87.4$ & $87.5$ & $\bf{70.1}$ & $\bf{70.1}$ & $\bf{70.1}$\\
    &  & $(0.00477)$ & $(0.00162)$ & $(0.0033)$ &  &  &  &  &  &  &  &  & \\
    & 100 & $0.0382$ & $\bf{0.0311}$ & $0.0578$ & $\bf{510}$ & $\bf{510}$ & $\bf{510}$ & $\bf{94.3}$ & $94.4$ & $94.5$ & $\bf{75.3}$ & $\bf{75.3}$ & $\bf{75.3}$\\
    &  & $(0.002)$ & $(0.000301)$ & $(0.00179)$ &  &  &  &  &  &  &  &  & \\
    & 200 & $0.138$ & $\bf{0.0991}$ & $0.205$ & $\bf{471}$ & $\bf{471}$ & $\bf{471}$ & $\bf{92.2}$ & $\bf{92.2}$ & $\bf{92.2}$ & $\bf{71}$ & $\bf{71}$ & $\bf{71}$\\
    &  & $(0.00662)$ & $(0.00138)$ & $(0.00213)$ &  &  &  &  &  &  &  &  & \\
    & 400 & $0.565$ & $\bf{0.464}$ & $0.862$ & $\bf{501}$ & $\bf{501}$ & $\bf{501}$ & $\bf{96.9}$ & $97$ & $97$ & $\bf{79.1}$ & $\bf{79.1}$ & $\bf{79.1}$\\
    &  & $(0.012)$ & $(0.000493)$ & $(0.0184)$ &  &  &  &  &  &  &  &  & \\
    2 & 50 & $1.71$ & $\bf{0.693}$ & $16$ & $122$ & $126$ & $\bf{118}$ & $85.9$ & $\bf{85.3}$ & $85.5$ & $72.2$ & $73.1$ & $\bf{70.8}$\\
    &  & $(0.00334)$ & $(0.00341)$ & $(0.00511)$ &  &  &  &  &  &  &  &  & \\
    & 100 & $11.8$ & $\bf{3.33}$ & $66.6$ & $\bf{162}$ & $163$ & $\bf{162}$ & $99.4$ & $\bf{98}$ & $98.4$ & $95.4$ & $95.8$ & $\bf{95.1}$\\
    &  & $(0.0189)$ & $(0.00704)$ & $(0.0844)$ &  &  &  &  &  &  &  &  & \\
    & 200 & $51.4$ & $\bf{14.1}$ & $230$ & $\bf{233}$ & $234$ & $\bf{233}$ & $98.1$ & $\bf{94.2}$ & $96.8$ & $123$ & $123$ & $\bf{122}$\\
    &  & $(0.0701)$ & $(0.00805)$ & $(0.273)$ &  &  &  &  &  &  &  &  & \\
    & 400 & $200$ & $\bf{50.3}$ & $917$ & $239$ & $239$ & $\bf{238}$ & $94.3$ & $\bf{90.8}$ & $91.9$ & $\bf{140}$ & $\bf{140}$ & $\bf{140}$\\
    &  & $(1.06)$ & $(0.299)$ & $(1.48)$ &  &  &  &  &  &  &  &  & \\
    10 & 50 & $0.19$ & $\bf{0.00722}$ & $0.196$ & $0.000891$ & $0.00109$ & $\bf{0.000838}$ & $0.821$ & $2.23$ & $\bf{0.488}$ & $\bf{8.59}$ & $8.63$ & $8.61$\\
    &  & $(0.00281)$ & $(9.99 \times 10^{-5})$ & $(0.00743)$ &  &  &  &  &  &  &  &  & \\
    & 100 & $0.854$ & $\bf{0.0644}$ & $0.873$ & $\bf{0.000937}$ & $0.00097$ & $0.000943$ & $0.154$ & $\bf{0}$ & $0.11$ & $10.3$ & $\bf{10.2}$ & $10.3$\\
    &  & $(0.0002)$ & $(0.000503)$ & $(0.00191)$ &  &  &  &  &  &  &  &  & \\
    & 200 & $3.77$ & $\bf{0.398}$ & $3.89$ & $\bf{0.000883}$ & $0.00099$ & $0.00101$ & $0.281$ & $\bf{0}$ & $0.292$ & $9.64$ & $\bf{9.63}$ & $9.65$\\
    &  & $(0.00486)$ & $(0.00888)$ & $(0.0132)$ &  &  &  &  &  &  &  &  & \\
    & 400 & $26.8$ & $\bf{3.17}$ & $27.6$ & $\bf{0.000992}$ & $0.000997$ & $0.000999$ & $0.185$ & $0.288$ & $\bf{0.176}$ & $9.41$ & $9.42$ & $\bf{9.39}$\\
    &  & $(0.405)$ & $(0.0183)$ & $(0.0488)$ &  &  &  &  &  &  &  &  & \\
    20 & 50 & $0.39$ & $\bf{0.00791}$ & $0.399$ & $0.000991$ & $0.00542$ & $\bf{0.00091}$ & $\bf{0.027}$ & $7.06$ & $0.0696$ & $9.77$ & $\bf{9.34}$ & $9.57$\\
    &  & $(0.00132)$ & $(2.6 \times 10^{-5})$ & $(0.00193)$ &  &  &  &  &  &  &  &  & \\
    & 100 & $1.46$ & $\bf{0.0684}$ & $1.58$ & $0.000995$ & $\bf{0.000965}$ & $0.000996$ & $0.0308$ & $\bf{0}$ & $\bf{0}$ & $9.13$ & $9.22$ & $\bf{9}$\\
    &  & $(0.000164)$ & $(0.000174)$ & $(0.0329)$ &  &  &  &  &  &  &  &  & \\
    & 200 & $7.7$ & $\bf{0.414}$ & $7.78$ & $0.000984$ & $0.00113$ & $\bf{0.000961}$ & $\bf{0.251}$ & $1.74$ & $0.438$ & $\bf{9.6}$ & $9.61$ & $\bf{9.6}$\\
    &  & $(0.166)$ & $(0.00109)$ & $(0.0105)$ &  &  &  &  &  &  &  &  & \\
    & 400 & $30.2$ & $\bf{3.03}$ & $30$ & $\bf{0.000997}$ & $0.00105$ & $0.00142$ & $\bf{0.0921}$ & $1.68$ & $0.646$ & $\bf{10}$ & $10.2$ & $10.2$\\
    &  & $(0.0275)$ & $(0.00605)$ & $(0.0157)$ &  &  &  &  &  &  &  &  & \\
    \bottomrule
    \end{tabular}
    } 
  	\end{scriptsize}
  	\caption{\label{tab:ex2} Convex regression. Times are reported as averages over 3 replicates.}
  \end{table}

\subsubsection{Convex Clustering.}

To evaluate the performance of the different methods on convex clustering, we consider a mixture of simulated data and discriminant analysis data from the UCI Machine Learning Repository \citep{dua2019uci}. The simulated data in \texttt{gaussian300} consists of 3 Gaussian clusters generated from bivariate normal distributions with means $\bmu = (0.0, 0.0)^{t}$, $(2.0, 2.0)^{t}$, and $(1.8, 0.5)^{t}$, standard deviation $\sigma = 0.1$, and class sizes $n_{1} = 150, n_{2} = 50, n_{3} = 100$. This easy dataset is included to validate Algorithm \ref{alg:search} described later as a reasonable solution path heuristic. The data in \texttt{iris} and \texttt{zoo} are representative of clustering with purely continuous or purely discrete data, respectively. In these two datasets, samples with same class label form a cluster. Finally, the simulated data \texttt{spiral500} is a classic example that thwarts $k$-means clustering. Each algorithm is allotted a maximum of $10^{4}$ inner iterations to solve a $\rho$-penalized subproblem at level $\delta_{h} = 10^{-2}$.
The annealing schedule is set to $\rho(t) = \min\{10^{8}, 1.2^{t-1}\}$ over 100 outer iterations with $\delta_{d} = 10^{-5}$ and $\delta_{q} = 10^{-6}$.
\begin{algorithm}[tbp]
    \footnotesize
    \begin{algorithmic}[1]
        \Procedure{cvxclusterpath}{$\bX$, $s_{0}$, $s_{\mathrm{step}}$}
        \State $\bU \gets \bX$
		\Comment{Initialize centroid assignments.}
        \State $s \gets s_{0}$
		\Comment{Initialize sparsity level in $[0,1]$.}
		\State $k_{\mathrm{max}} \gets \binom{m}{2}$
		\Comment{Determine upper bound from samples $m$.}
		\While{$s < 1$}
			\State $k \gets \mathrm{round}((1-s) \times k_{\mathrm{max}})$
			\Comment{Set parameter for sparsity projection.}
            \State $\bU_{s} \gets \argmin h_{\rho,k}(\bU)$
			\Comment{Minimize with choice of $k$.}
			\State $c \gets \mathrm{count}(\bU_{s})$
			\Comment{Count the number of coalesced centroid differences.}
			\State $s_{\mathrm{proposal}} \gets c /k_{\mathrm{max}}$
            \Comment{Propose a new sparsity level based on satisfied constraints.}
			\If{$s_{\mathrm{proposal}} > s + s_{\mathrm{step}}$}
				\State $s \gets s_{\mathrm{proposal}}$
				\Comment{Accept proposal if it increases sparsity.}
			\Else
				\State $s \gets s + s_{\mathrm{step}}$
				\Comment{Otherwise move by a fixed amount.}
			\EndIf
        \EndWhile
        \EndProcedure
    \end{algorithmic}
    \caption{Search Heuristic}
    \label{alg:search}
\end{algorithm}

Because the number of clusters is usually unknown, we implement the search heuristic outlined in Algorithm \ref{alg:search}. The idea behind the heuristic is to gradually coerce clustering without exploring the full range of the hyperparameter $k$.
As one decreases the number of admissible nonzero centroid differences $k$ from $k_{\max}$ to $0$, sparsity ($1-k/k_{\max}$) in the columns of $\bU \bD$ increases to reflect coalescing centroid assignments.
Thus, Algorithm 1 generates a list of candidate clusters that can be evaluated by various measures of similarity \citep{vinh_information_2010}. 
For example, the adjusted Rand index (ARI) provides a reasonable measure of the distance to the ground truth in our examples as it accounts for both the number of identified clusters and cluster assignments.
We also report the related normed Mutual Information (NMI).
The ARI takes values on $[-1,1]$ whereas NMI appears on a $[0,1]$ scale.

ADMM, as implemented here, is not remotely competitive on these examples given its extremely long compute times and failure to converge in some instances. These times are only exacerbated by the search heuristic and therefore omit ADMM from this example.
The findings reported in Table \ref{tab:ex3} indicate the same accuracy for MM (using LSQR) and SD as measured by loss and distance to feasibility. Here we see that the combination of the proximal distance algorithms and the overall search heuristic (Algorithm \ref{alg:search}) yields perfect clusters in the \texttt{gaussian300} example on the basis of ARI and NMI. To its disadvantage, the search heuristic is greedy and generally requires tuning.
Both MM and SD achieve similar clusterings as indicated by ARI and NMI. Notably, SD generates candidate clusterings faster than MM.
\begin{table}[tbp]
	\centering
	\setlength{\tabcolsep}{2pt}
	\begin{scriptsize}
    \begin{tabular}{rrrrcccccccccc}
    \toprule
    & & &  & \multicolumn{2}{c}{Time (s)} & \multicolumn{2}{c}{Loss} & \multicolumn{2}{c}{Distance $\times 10^{5}$} & \multicolumn{2}{c}{ARI} & \multicolumn{2}{c}{NMI}\\
    \cmidrule(lr){5-6} \cmidrule(lr){7-8} \cmidrule(lr){9-10} \cmidrule(lr){11-12} \cmidrule(lr){13-14}
    dataset & features & samples & classes & MM & SD & MM & SD & MM & SD & MM & SD & MM & SD\\
    \midrule
    zoo & 16 & 101 & 7 & $95.6$ & $\bf{77.1}$ & $\bf{1600}$ & $\bf{1600}$ & $\bf{8.62}$ & $\bf{8.62}$ & $0.841$ & $\bf{0.848}$ & $0.853$ & $\bf{0.856}$\\
    & & &  & $(0.268)$ & $(2.02)$ &  &  &  &  & $(0)$ & $(0.0118)$ & $(0)$ & $(0.00256)$\\
    iris & 4 & 150 & 3 & $76.4$ & $\bf{62.8}$ & $\bf{596}$ & $\bf{596}$ & $\bf{8.8}$ & $\bf{8.8}$ & $\bf{0.575}$ & $\bf{0.575}$ & $\bf{0.734}$ & $\bf{0.734}$\\
    & & &  & $(0.129)$ & $(2)$ &  &  &  &  & $(0)$ & $(0)$ & $(0)$ & $(0)$\\
    gaussian300 & 2 & 300 & 3 & $190$ & $\bf{155}$ & $\bf{598}$ & $\bf{598}$ & $\bf{8.98}$ & $\bf{8.98}$ & $\bf{1}$ & $\bf{1}$ & $\bf{1}$ & $\bf{1}$\\
    & & &  & $(0.173)$ & $(0.177)$ &  &  &  &  & $(0)$ & $(0)$ & $(0)$ & $(0)$\\
    spiral500 & 2 & 500 & 2 & $715$ & $\bf{561}$ & $\bf{998}$ & $\bf{998}$ & $\bf{8.98}$ & $\bf{8.98}$ & $\bf{0.133}$ & $\bf{0.133}$ & $\bf{0.366}$ & $\bf{0.366}$\\
    & & &  & $(18.5)$ & $(0.501)$ &  &  &  &  & $(0)$ & $(0)$ & $(0)$ & $(0)$\\
    \bottomrule
    \end{tabular}
	\end{scriptsize}
	\caption{\label{tab:ex3} Convex clustering. Times reflect the total time spent generating candidate clusterings using Algorithm \ref{alg:search}. Additional metrics correspond to the optimal clustering on the basis of maximal ARI. Time and clustering indices are averaged over 3 replicates with standard deviations reported in parentheses.
	} 
\end{table}

\subsubsection{Total Variation Image Denoising.}

To evaluate our denoising algorithm, we consider two standard test images, \texttt{cameraman} and \texttt{peppers\_gray}. White noise with $\sigma = 0.2$ is applied to an image and then reconstructed using our proximal distance algorithms. Only MM and SD are tested with a maximum of $100$ outer and $10^{4}$ inner iterations, respectively, and convergence thresholds $\delta_{h} = 10^{-1}$, $\delta_{d} = 10^{-1}$, and $\delta_{q} = 10^{-6}$. A moderate schedule $\rho(t) = \{10^{8}, 1.5^{t-1}\}$ performs well even with such lax convergence criteria. Table \ref{tab:ex4} reports convergence metrics and image quality indices, MSE and Peak Signal-to-Noise Ratio (PSNR). Timings reflect the total time spent generating solutions, starting from a 0\% reduction in the total variation of the input image $\bU$ up to 90\% reduction in increments of 10\%.
Explicitly, we take $\gamma_{0} = \mathrm{TV}_{1}(\bU)$ and vary the control parameter $\gamma = (1-s) \gamma_{0}$ with $s \in [0,1]$ to control the strength of denoising.
Figure \ref{fig:images} depicts the original and reconstructed images along the solution path.
\begin{table}[!b]
	\centering
	\setlength{\tabcolsep}{2pt}
	\begin{scriptsize}
    \begin{tabular}{rrrcccccccccc}
    \toprule
    & & & \multicolumn{2}{c}{Time (s)} & \multicolumn{2}{c}{Loss} & \multicolumn{2}{c}{Distance $\times 10^{3}$} & \multicolumn{2}{c}{MSE $\times 10^{5}$} & \multicolumn{2}{c}{PSNR}\\
    \cmidrule(lr){4-5} \cmidrule(lr){6-7} \cmidrule(lr){8-9} \cmidrule(lr){10-11} \cmidrule(lr){12-13}
    image & width & height & MM & SD & MM & SD & MM & SD & MM & SD & MM & SD\\
    \midrule
    cameraman & 512 & 512 & $557$ & $\bf{181}$ & $\bf{8090}$ & $\bf{8090}$ & $\bf{93.1}$ & $\bf{93.1}$ & $\bf{284}$ & $\bf{284}$ & $\bf{25.5}$ & $\bf{25.5}$\\
    & &  & $(5.01)$ & $(0.906)$ &  &  &  &  &  &  &  & \\
    peppers\_gray & 512 & 512 & $553$ & $\bf{183}$ & $\bf{8020}$ & $\bf{8020}$ & $92.6$ & $\bf{92.5}$ & $\bf{290}$ & $\bf{290}$ & $\bf{25.4}$ & $\bf{25.4}$\\
    & &  & $(6.28)$ & $(2.11)$ &  &  &  &  &  &  &  & \\
    \bottomrule
    \end{tabular}
	\end{scriptsize}
	\caption{\label{tab:ex4} Image denoising. Times reflect the total time spent generating candidate images, averaged over 3 replicates, ultimately achieving 90\% reduction in total variation}
\end{table}

\begin{figure}[p]
	\centering
    \includegraphics[width=0.3\textwidth]{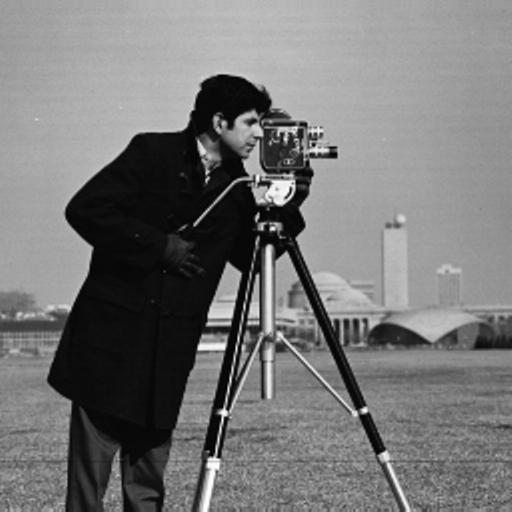}%
    \includegraphics[width=0.3\textwidth]{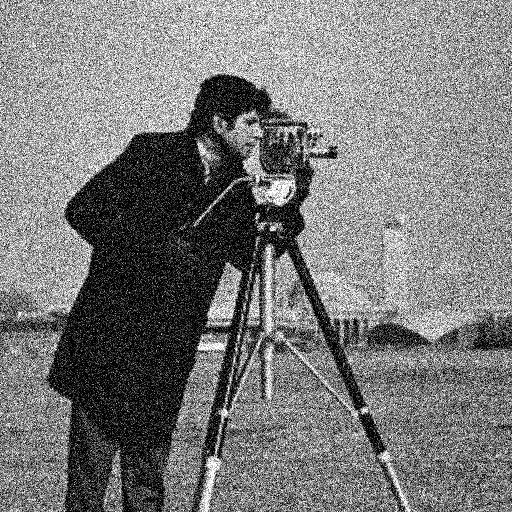}%
    \includegraphics[width=0.3\textwidth]{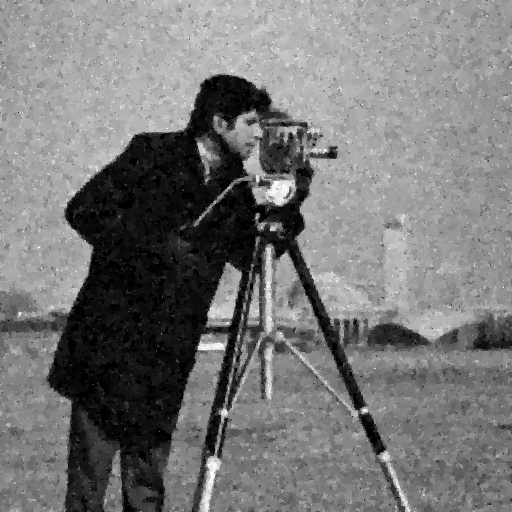}
    \includegraphics[width=0.3\textwidth]{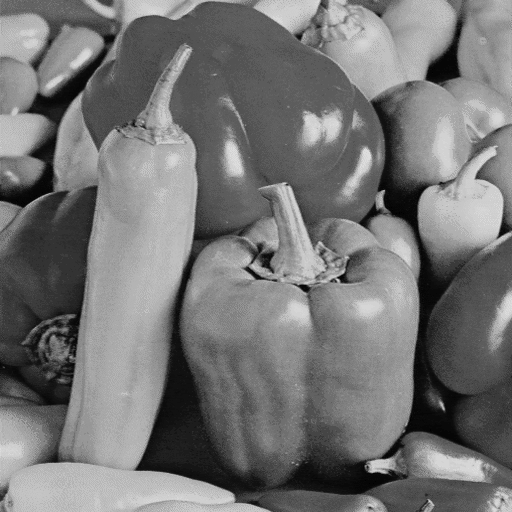}%
    \includegraphics[width=0.3\textwidth]{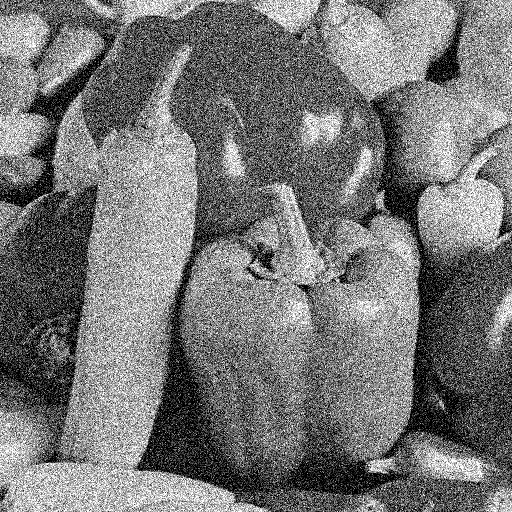}%
    \includegraphics[width=0.3\textwidth]{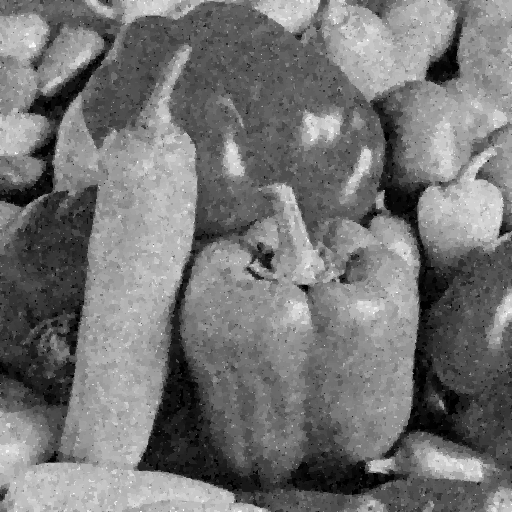}
    \label{fig:images}
	\caption{
		Sample images along the solution path of the search heuristic.
		Images are arranged from left to right as follows: reference, noisy input, and 90\% reduction.
	}
\end{figure}

\newpage

\subsubsection{Projection of a Matrix to a Good Condition Number.}

We generate base matrices $\bM \in \mathbb{R}^{p \times p}$ as random correlation matrices using MatrixDepot.jl \citep{zhang_matrix_2016}, which relies on Davies' and Higham's refinement \citep{davies_numerically_2000} of the Bendel-Mickey algorithm \citep{bendel_population_1978}. Simulations generate matrices with condition numbers c$(\bM)$ in the set $\{119,1920,690\}$. Our subsequent analyses target condition number decreases by a factor $a$ such $a \in \{2, 4, 16, 32\}$.
Each algorithm is allotted a maximum of $200$ outer and $10^{4}$ inner iterations, respectively with choices $\delta_{h} = 10^{-3}$, $\delta_{d} = 10^{-2}, \delta_{q} = 10^{-6}$, and $\rho(t) = 1.2^{t-1}$. Table \ref{tab:ex5} summarizes the performance of the three algorithms. The quality of approximate solutions is similar across MM, SD, and ADMM in terms of loss, distance, and final condition number metrics.
Interestingly, the MM approach requires less time to deliver solutions of comparable quality to SD solutions as the size $p$ of the input matrix $\bM \in \mathbb{R}^{p \times p}$ increases.
\begin{table}[tbp]
	\centering
	\setlength{\tabcolsep}{2pt}
	\begin{scriptsize}
	\resizebox{\textwidth}{!}{%
		\begin{tabular}{rrrcccccccccccc}
        \toprule
        & &  & \multicolumn{3}{c}{Time (ms)} & \multicolumn{3}{c}{Loss $\times 10^{3}$} & \multicolumn{3}{c}{Distance $\times 10^{5}$} & \multicolumn{3}{c}{$c(\bX)$}\\
        \cmidrule(lr){4-6} \cmidrule(lr){7-9} \cmidrule(lr){10-12} \cmidrule(lr){13-15}
        $p$ & $c(\bM)$ & $a$ & MM & SD & ADMM & MM & SD & ADMM & MM & SD & ADMM & MM & SD & ADMM\\
        \midrule
        10 & 119 & 2 & $0.576$ & $\bf{0.391}$ & $0.582$ & $0.549$ & $\bf{0.548}$ & $0.549$ & $41.1$ & $41.1$ & $\bf{40.6}$ & $\bf{59.4}$ & $\bf{59.4}$ & $\bf{59.4}$\\
        & &  & $(0.104)$ & $(0.00629)$ & $(0.172)$ &  &  &  &  &  &  &  &  & \\
        & & 4 & $0.413$ & $\bf{0.359}$ & $0.543$ & $\bf{4.92}$ & $\bf{4.92}$ & $\bf{4.92}$ & $236$ & $\bf{234}$ & $236$ & $\bf{29.7}$ & $\bf{29.7}$ & $\bf{29.7}$\\
        & &  & $(0.00321)$ & $(0.00129)$ & $(0.101)$ &  &  &  &  &  &  &  &  & \\
        & & 16 & $0.439$ & $\bf{0.424}$ & $0.535$ & $\bf{130}$ & $\bf{130}$ & $\bf{130}$ & $\bf{926}$ & $929$ & $927$ & $\bf{7.44}$ & $\bf{7.44}$ & $\bf{7.44}$\\
        & &  & $(0.00319)$ & $(0.00427)$ & $(0.00395)$ &  &  &  &  &  &  &  &  & \\
        & & 32 & $0.662$ & $\bf{0.496}$ & $0.824$ & $\bf{821}$ & $\bf{821}$ & $\bf{821}$ & $984$ & $\bf{983}$ & $\bf{983}$ & $\bf{3.72}$ & $\bf{3.72}$ & $\bf{3.72}$\\
        & &  & $(0.00336)$ & $(0.00357)$ & $(0.00459)$ &  &  &  &  &  &  &  &  & \\
        100 & 1920 & 2 & $\bf{37.7}$ & $51.1$ & $149$ & $\bf{0.00119}$ & $\bf{0.00119}$ & $\bf{0.00119}$ & $0.208$ & $0.208$ & $\bf{0.204}$ & $\bf{960}$ & $\bf{960}$ & $\bf{960}$\\
        & &  & $(2.89)$ & $(3.77)$ & $(0.972)$ &  &  &  &  &  &  &  &  & \\
        & & 4 & $\bf{17.1}$ & $26.5$ & $66.4$ & $\bf{0.0107}$ & $\bf{0.0107}$ & $\bf{0.0107}$ & $\bf{0.845}$ & $\bf{0.845}$ & $0.877$ & $\bf{480}$ & $\bf{480}$ & $\bf{480}$\\
        & &  & $(0.164)$ & $(2.8)$ & $(0.485)$ &  &  &  &  &  &  &  &  & \\
        & & 16 & $\bf{10.1}$ & $11.9$ & $35.3$ & $\bf{0.436}$ & $\bf{0.436}$ & $\bf{0.436}$ & $\bf{17.8}$ & $17.9$ & $17.9$ & $\bf{120}$ & $\bf{120}$ & $\bf{120}$\\
        & &  & $(0.301)$ & $(0.0472)$ & $(0.437)$ &  &  &  &  &  &  &  &  & \\
        & & 32 & $\bf{26}$ & $33.4$ & $74$ & $\bf{3.26}$ & $\bf{3.26}$ & $\bf{3.26}$ & $\bf{95.2}$ & $\bf{95.2}$ & $\bf{95.2}$ & $\bf{60}$ & $\bf{60}$ & $\bf{60}$\\
        & &  & $(2.23)$ & $(0.49)$ & $(0.893)$ &  &  &  &  &  &  &  &  & \\
        1000 & 59400 & 2 & $\bf{60.3}$ & $75.3$ & $157$ & $\bf{1.15 \times 10^{-6}}$ & $\bf{1.15 \times 10^{-6}}$ & $\bf{1.15 \times 10^{-6}}$ & $\bf{0}$ & $\bf{0}$ & $\bf{0}$ & $29700$ & $29700$ & $\bf{29600}$\\
        & &  & $(0.83)$ & $(1.35)$ & $(3.98)$ &  &  &  &  &  &  &  &  & \\
        & & 4 & $\bf{57.4}$ & $71.8$ & $131$ & $\bf{1.04 \times 10^{-5}}$ & $\bf{1.04 \times 10^{-5}}$ & $\bf{1.04 \times 10^{-5}}$ & $\bf{0}$ & $\bf{0}$ & $\bf{0}$ & $\bf{14800}$ & $\bf{14800}$ & $\bf{14800}$\\
        & &  & $(3.37)$ & $(0.882)$ & $(8.44)$ &  &  &  &  &  &  &  &  & \\
        & & 16 & $\bf{55.6}$ & $71.9$ & $152$ & $\bf{0.000258}$ & $\bf{0.000258}$ & $0.000261$ & $\bf{0}$ & $\bf{0}$ & $\bf{0}$ & $3710$ & $3710$ & $\bf{3690}$\\
        & &  & $(3.3)$ & $(2.05)$ & $(38.8)$ &  &  &  &  &  &  &  &  & \\
        & & 32 & $\bf{23200}$ & $30000$ & $87400$ & $\bf{0.0011}$ & $\bf{0.0011}$ & $0.00111$ & $0.11$ & $0.11$ & $\bf{0}$ & $1860$ & $1860$ & $\bf{1850}$\\
        & & & $(855)$ & $(337)$ & $(1340)$ &  &  &  &  &  &  &  &  & \\
        \bottomrule
        \end{tabular}
        } 
	\end{scriptsize}
	\caption{\label{tab:ex5} Condition number experiments.  Here $c(\bM)$  is the condition number of the input matrix, $a$ is the decrease factor, and $c(\bX)$  is the condition number of the solution.}
\end{table}

\section{Discussion}

We now recapitulate the main findings of our numerical experiments. Tables \ref{tab:ex1} through \ref{tab:ex5} show a consistent pattern of superior speed by the steepest descent (SD) version of the proximal distance algorithm. This is hardly surprising since unlike ADMM and MM, SD avoids solving a linear system at each iteration. SD's speed advantage tends to persist even when the linear system can be solved exactly. The condition number example summarized in Table \ref{tab:ex5} is an exception to this rule. Here the MM updates leverage a very simple matrix inverse. MM is usually faster than ADMM. We attribute MM's performance edge to the extra matrix-vector multiplications involving the fusion matrix $\bD$ required by ADMM. In fairness, ADMM closes the speed gap and matches MM on convex regression.

The choice of annealing schedule can strongly impact the quality of solutions. Intuitively, driving the gradient norm $\|\nabla h_{\rho}(\bx)\|$ to nearly 0 for a given $\rho$ keeps the proximal distance methods on the correct annealing path and yields better approximations. Provided the penalized objective is sufficiently smooth, one expects the solution $\bx_{\rho} \approx \argmin h_{\rho}(\bx)$ to be close to the solution  $\bx_{\rho^{\prime}} \approx \argmin h_{\rho^{\prime}}(\bx)$ when the ratio $\rho^{\prime}/ \rho >1$ is not too large. Thus, choosing a conservative $\delta_h$ for the convergence criterion $\|\nabla h_{\rho}(\bx)\| \le \delta_h$ may guard against a poorly specified annealing schedule.
Quantifying sensitivity of intermediate solutions $\bx_{\rho}$ with respect to $\rho$ is key in avoiding an increase in inner iterations per subproblem; for example, as observed in Figure \ref{fig:converge}. Given the success of our practical annealing recommendation to overcome the unfortunate coefficients in Propositions \ref{propositionc} and \ref{proposition0}, this topic merits further consideration in future work.

In practice, it is sometimes unnecessary to impose strict convergence criteria on the proximal distance iterates.
It is apparent that the convergence criteria on convex clustering and image denoising are quite lax compared to choices in other examples, specifically in terms of $\delta_{h}$.
Figure \ref{fig:converge} suggests that most of the work in metric projection involves driving the distance penalty downhill rather than in fitting the loss. Surpisingly, Table \ref{tab:ex3} shows that our strict distance criterion  $\delta_{d} = 10^{-5}$ in clustering is achieved. This implies $\dist(\bD \bx, S)^{2} \le 10^{-10}$ on the selected solutions with $\rho \le 10^{8}$, yet we only required $\|\nabla h_{\rho}(\bx)\| \le 10^{-2}$ on each subproblem. Indeed, not every problem may benefit from precise solution estimates. The image processing example underscores this point as we are able to recover denoised images with the choices $\delta_{h} = \delta_{d} = 10^{-1}$. Problems where patterns or structure in solutions are of primary interest may stand to benefit from relaxed convergence criteria.

Our proximal distance method, as described in Algorithm \ref{alg:pdi}, enjoys several advantages. First, fusion constraints fit naturally in the proximal distance framework. Second, proximal distances enjoy the descent property. Third, there is a nearly optimal step size for gradient descent when second-order information is available on the loss.  Fourth, proximal distance algorithms are competitive if not superior to ADMM on many problems. Fifth, proximal distance algorithms like iterative hard thresholding rely on set projection and are therefore helpful in dealing with hard sparsity constraints. The main disadvantages of the proximal distance methods are (a) the overall slow convergence due to the loss of curvature information on the distance penalty and (b) the need for a reasonable annealing schedule. In practice, a little experimentation can yield a reasonable schedule for an entire class of problems. Many competing methods are only capable of dealing with soft constraints imposed by the lasso and other convex penalties. To their detriment, soft constraints often entail severe parameter shrinkage and lead to an excess of false positives in model selection.

Throughout this manuscript we have stressed the flexibility of the proximal distance framework in dealing with a wide range of constraints as a major strength. From our point of view, proximal distance iteration adequately approximates feasible, locally optimal solutions to constrained optimization problems for well behaved constraint sets, for instance convex sets or semialgebraic sets. Combinatorially complex constraints or erratic loss functions can cause difficulties for the proximal distance methods. The quadratic distance penalty $\dist(\bD \bx, S)^{2}$ is usually not an issue, and projection onto the constraint should be fast. Poor loss functions may either lack second derivatives or may possess a prohibitively expensive and potentially ill-conditioned Hessian $d^{2} f(\bx)$. In this setting techniques such as coordinate descent and regularized and quasi-Newton methods are viable alternatives for minimizing the surrogate $g_{\rho}(\bx \mid \bx_{n})$ generated by distance majorization. In any event, it is crucial to design a surrogate that renders each subproblem along the annealing path easy to solve. This may entail applying additional majorizations in $f(\bx)$. Balanced against this possibility is the sacrifice of curvature information with each additional majorization.

We readily acknowledge that other algorithms may perform better than MM and proximal distance algorithms on specific problems. The triangle fixing algorithm for metric projection is a case in point \citep{brickell_metric_2008}, as are the numerous denoising algorithms based on the $\ell_{1}$ norm. This objection obscures the generic utility of the proximal distance principle. ADMM can certainly be beat on many specific problems, but nobody seriously suggests that it be rejected across the board. Optimization, particularly constrained optimization, is a fragmented subject, with no clear winner across problem domains. Generic methods serve as workhorses, benchmarks, and backstops.

As an aside, let us briefly note that ADMM can be motivated by the MM principle, which is the same idea driving proximal distance algorithms. The optimal pair $(\bx,\by)$ and $\blambda$ furnishes a stationary point of the Lagrangian. Because the Lagrangian is linear in $\blambda$, its maximum for fixed $(\bx,\by)$ is $\infty$. To correct this defect, one can add a viscosity minorization to the Lagrangian. This produces the modified Lagrangian
\begin{eqnarray*}
\mathcal{L}_\mu(\bx,\by, \blambda) & \!\! = \!\! & f(\bx)+ g(\by)
+ \mu \blambda^t(\bD\bx-\by)+\frac{\mu}{2} \|\bD\bx -\by\|^2 
- \frac{\alpha}{2}\|\blambda-\blambda_n \|^2.
\end{eqnarray*}
The penalty term has no impact on the $\bx$ and $\by$ updates. However, the MM  update for $\blambda$ is determined by the stationary condition
\begin{eqnarray*}
{\bf 0} & = & \mu (\bD\bx_{n+1}-\by_{n+1})-\alpha(\blambda-\blambda_n),
\end{eqnarray*}
so that
\begin{eqnarray*}
\blambda_{n+1} & = & \blambda_n+ \frac{\mu}{\alpha} (\bD\bx_{n+1}-\by_{n+1}).
\end{eqnarray*}
The choice $\alpha =1$ gives the standard ADMM update. Thus, the ADMM algorithm alternates decreasing and increasing the Lagrangian in a search for the saddlepoint represented by the optimal trio $(\bx,\by,\blambda)$.

In closing we would like to draw the reader's attention to some generalizations of the MM principle and connections to other well-studied algorithm classes. For instance, a linear fusion constraint $\bD\bx \in S$ can in principle by replaced by a nonlinear fusion constraint $M(\bx)\in S$. The objective and majorizer are then
\begin{eqnarray*}
h_\rho(\bx) & = &  f(\bx)+\frac{\rho}{2}\dist[M(\bx),S]^2 \\
g(\bx \mid \bx_n) & = & f(\bx)+ \frac{\rho}{2}\|M(\bx)-\mathcal{P}_S[M(\bx_n)]\|^2.
\end{eqnarray*}
The objective has gradient $\bg=\nabla f(\bx)+\rho dM(\bx)^t\{M(\bx)-\mathcal{P}_S[M(\bx)]\}$. The second differential of the majorizer is approximately $d^2 f(\bx)+\rho dM(\bx)^td M(\bx)$ for $M(\bx)$ close to $\mathcal{P}_S[M(\bx)]$. Thus, gradient descent can be implemented with step size 
\begin{eqnarray*}
\gamma & = & \frac{\|\bg_n\|^2}{\bg_n^t d^2f(\bx_n) \bg_n + \rho \|dM(\bx_n)\bg_n\|^2}, 
\end{eqnarray*}
assuming the denominator is positive.

Algebraic penalties such as $\|g(\bx)\|^2$ reduce to distance penalties with constraint set $\{\bf 0\}$. The corresponding projection operator sends any vector $\by$ to $\bf 0$ so that the algebraic penalty $\|g(\bx)\|^2 = \dist[g(\bx), \{\bf 0\}]^2$. This observation is pertinent to constrained least squares with $g(\bx)=\bd-\bC\bx$ \citep{golub1996matrix}. The proximal distance surrogate can be expressed as
\begin{eqnarray*}
\frac{1}{2}\|\by-\bA\bx\|^2 +\frac{\rho}{2}\|\bd-\bC\bx\|^2 & = &
\frac{1}{2}\left\|\begin{pmatrix} \by \\ \sqrt{\rho}\bd  \end{pmatrix} - \begin{bmatrix} \bA \\
\sqrt{\rho}\bC \end{bmatrix} \bx \right\|^2
\end{eqnarray*}
and minimized by standard least squares algorithms. No annealing is necessary. Inequality constraints $g(\bx) \le {\bf 0}$ behave somewhat differently. The proximal distance majorization
$\dist[g(\bx),\mathbb{R}^m_-]^2 \le \|g(\bx)-\mathcal{P}_{\mathbb{R}^m_-}[g(\bx_n)]\|^2$ is not the same as the Beltrami quadratic penalty $g(\bx)_+^2$ \citep{beltrami1970algorithmic}. However, the standard majorization \citep{lange2016mm}
\begin{eqnarray*}
g(\bx)_+^2 & \le & \|g(\bx)-\mathcal{P}_{\mathbb{R}^m_-}[g(\bx_n)]\|^2
\end{eqnarray*}
brings them back into alignment.

\section{Proofs}
\label{sec:proofs}

In this section we provide proofs for the convergence results discussed in Section \ref{sec:analysis} and Section \ref{sec:nonconvex} for the convex and noncovex cases, respectively.

\subsection{Proposition \ref{propositiona}   }
\begin{proof}
Without loss of generality we can translate the coordinates so that $\by={\bf 0}$. Let $B$ be the unit sphere $\{\bx: \|\bx\|= 1\}$. Our first aim is to show that $h_\rho(\bx) > f({\bf 0})$ throughout $B$. Consider the set $B \cap T$, which is possibly empty. On this set the infimum $b$ of $f(\bx)$ is attained, so $b>f({\bf 0})$ by assumption. The set $B \setminus T$ will be divided into two regions,
a narrow zone adjacent to $T$ and the remainder. Now let us show that there exists a $\delta > 0$ such that $h_\rho(\bx) \ge f(\bx)\ge f({\bf 0})+\delta$ for all $\bx \in B$ with $\dist(\bD\bx, S) \le \delta$. If this is not so, then there exists a sequence $\bx_n \in B$ with $f(\bx_n) < f({\bf 0})+\frac{1}{k}$ and $\dist(\bD\bx_n,S) \le \frac{1}{k}$. By compactness, some subsequence of $\bx_n$ converges to $\bz \in B \cap T$ with $f(\bz)\le f({\bf 0})$, contradicting the uniqueness of $\by$. Finally, let $a = \min_{\bx \in B}f(\bx)$. To deal with the remaining region take $\rho$ large enough so that $a+\frac{\rho}{2}\delta^2 > f({\bf 0})$. For such $\rho$, $h_\rho(\bx) > f({\bf 0})$ everywhere on $B$. It follows that on the unit ball $\{\bx: \|\bx\| \le 1\}$, $h_\rho(\bx)$ is minimized at an interior point. Because $h_\rho(\bx)$ is convex, a local minimum is necessarily a global minimum.

To show that the objective $h_\rho(\bx)$ is coercive, it suffices to show that it is coercive along every ray $\{t\bv:t\ge 0, \; \|\bv\|=1\}$ \citep{lange2016mm}. The convex function $r(t) = h_\rho(t\bv)$ satisfies $r(t) \ge r(1)+r_{+}'(1)(t-1)$. Because $r(0) < r(1)$, the point $1$ is on the upward slope of $r(t)$, and the one-sided derivative $r_{+}'(1)>0$. Coerciveness follows from this observation.
\end{proof}

\subsection{Proposition \ref{propositionb}  }

\begin{proof} The first assertion follows from the bound $g_\rho(\bx \mid \bx_n) \ge h_\rho(\bx)$. To prove the second assertion, we note that it suffices to prove the existence of some constant $\rho>0$ such that the matrix $\bA+\rho \bD^t\bD$ is positive definite \citep{debreu1952definite}. If no choice of $\rho$ renders $\bA+\rho \bD^t\bD$ positive definite, then there is a sequence of unit vectors $\bu_{m}$ and a sequence of scalars $\rho_{m}$ tending to $\infty$ such that 
\begin{eqnarray}
\bu_{m}^t \bA \bu_{m}+\rho_{m} \bu_{m}^t \bD^t\bD \bu_{m} & \le & 0 . \label{hestenes}
\end{eqnarray}
By passing to a subsequence if needed, we may assume that the sequence
$\bu_{m}$ converges to a unit vector $\bu$.  On the one hand, because $\bD^t\bD$ is positive semidefinite, inequality (\ref{hestenes}) compels the conclusions $\bu_{m}^t \bA \bu_{m} \le 0$, which must carry over to the limit. On the other hand, dividing inequality (\ref{hestenes}) by $\rho_{m}$ and taking limits imply $\bu^t \bD^t \bD \bu \le 0$ and therefore $\|\bD \bu\| = 0$.  Because the limit vector $\bu$ violates the condition $\bu^t \bA \bu>0$, the required $\rho >0$ exists.
\end{proof}

\subsection{Proposition \ref{propositionc}  }

\begin{proof}
Systematic decrease of the iterate values $h_\rho(\bx_n)$ is a consequence of the MM principle. The existence of
$\bz_\rho$ follows from Proposition \ref{propositiona}. To prove the stated bound, first observe that the function
$g_\rho(\bx \mid \bx_n)-\frac{\rho}{2}\|\bD\bx\|^2$ is convex, being the sum of the convex function $f(\bx)$ and a linear function. Because
$\nabla g_\rho(\bx_{n+1} \mid \bx_n)^t(\bx-\bx_{n+1}) \ge \bf 0$ for any $\bx$ in $S$, the supporting hyperplane inequality implies that
\begin{eqnarray*}
g_{\rho}(\bx \mid \bx_n)-\frac{\rho}{2}\|\bD\bx\|^2
& \ge & g_{\rho}(\bx_{n+1} \mid \bx_n)-\frac{\rho}{2}\|\bD\bx_{n+1}\|^2 \\
&  & -\rho  \bx_{n+1}^t\bD^t\bD(\bx-\bx_{n+1}) ,
\end{eqnarray*}
or equivalently
\begin{eqnarray}
g_{\rho}(\bx \mid \bx_n) & \ge & g_{\rho}(\bx_{n+1} \mid \bx_n) 
+ \frac{\rho}{2}  \|\bD(\bx-\bx_{n+1})\|^2. \label{error_bd1}
\end{eqnarray}
Now note that the difference
\begin{eqnarray*}
d(\bx \mid \by) & = &  \frac{1}{2}\|\bx-\mathcal{P}(\by)\|^2 -\frac{1}{2}\|\bx -\mathcal{P}(\bx)\|^2
\end{eqnarray*}
has gradient
\begin{eqnarray*}
\nabla d(\bx \mid \by) & = & \mathcal{P}(\bx) - \mathcal{P}(\by).
\end{eqnarray*}
Because $\mathcal{P}(\bx)$ is non-expansive, the gradient $\nabla d(\bx \mid \by)$ is Lipschitz with constant $1$.
The tangency conditions $d(\by \mid \by)=0$ and $\nabla d(\by\mid \by)= {\bf 0}$ therefore yield 
\begin{eqnarray}
d(\bx \mid \by ) & \le & d(\by \mid \by) +
 \nabla d(\by \mid \by)^t(\bx-\by ) +\frac{1}{2}\|\bx-\by\|^2 \nonumber \\
& = &\frac{1}{2}\|\bx-\by\|^2  \label{error_bd2}
\end{eqnarray}
for all $\bx$.
At a minimum $\bz_\rho$ of $h_\rho(\bx)$, combining inequalities (\ref{error_bd1}) and (\ref{error_bd2}) gives
\begin{eqnarray*}
&  & h_\rho(\bx_{n+1})+\frac{\rho}{2} \|\bD(\bz_\rho-\bx_{n+1})\|^2 \\
& \le & g_\rho(\bx_{n+1} \mid \bx_n)+\frac{\rho}{2}\|\bD(\bz_\rho-\bx_{n+1})\|^2 \\
& \le & g_\rho(\bz_\rho \mid \bx_n) \\
& = & h_\rho(\bz_\rho)-\frac{\rho}{2}\|\bD\bz_\rho-\mathcal{P}(\bD\bz_\rho)\|^2+\frac{\rho}{2}\|\bD\bz_\rho-\mathcal{P}(\bD\bx_n)\|^2
\\
& = &  h_\rho(\bz_\rho) + \rho d(\bD\bz_\rho \mid \bD\bx_n) \\
& \le & h_\rho(\bz_\rho)+\frac{\rho}{2}\|\bD\bz_\rho-\bD\bx_{n}\|^2 .
\end{eqnarray*}
Adding the result 
\begin{eqnarray*}
h_\rho(\bx_{n+1})-h_\rho(\bz_\rho) & \le & 
\frac{\rho}{2}\Big[\|\bD(\bz_\rho-\bx_{n})\|^2
 -\|\bD(\bz_\rho-\bx_{n+1})\|^2\Big]
\end{eqnarray*}
over $n$ and invoking the descent property $h_\rho(\bx_{n+1}) \le h_\rho(\bx_n)$, telescoping produces the desired error bound
\begin{eqnarray*}
h_\rho(\bx_{n+1})-h_\rho(\bz_\rho) & \le &
\frac{\rho}{2(n+1)}  \Big[\|\bD(\bz_\rho-\bx_{0})\|^2 
-\|\bD(\bz_\rho-\bx_{n+1})\|^2\Big] \\
& \le & \frac{\rho}{2(n+1)}\|\bD(\bz_\rho-\bx_{0})\|^2.
\end{eqnarray*}
 This is precisely the asserted bound.
\end{proof}

\subsection{Proposition \ref{proposition0}  }

\begin{proof}
The existence and uniqueness of $\bz_\rho$ are obvious. The remainder of the proof hinges on the facts that $h_\rho(\bx)$ is $\mu$-strongly convex and the surrogate $g_\rho (\bx \mid \bw)$ is $L+\rho\|\bD\|^2$-smooth for all $\bw$. The latter assertion follows from 
\begin{eqnarray*}
\nabla g_\rho (\bx \mid \bw) - \nabla g_\rho (\by \mid \bw)  
& = & \nabla f(\bx)-\nabla f(\by) + \rho \bD^t\bD(\bx-\by).
\end{eqnarray*}
These facts together with $\nabla g_\rho(\bz_\rho \mid \bz_\rho) =\nabla h_\rho(\bz_\rho)= {\bf 0}$ imply 
\begin{eqnarray}
h_\rho(\bx)-h_\rho(\bz_\rho) & \le & g_\rho (\bx \mid \bz_\rho) - g_\rho (\bz_\rho \mid \bz_\rho)  
\nonumber \\
& \le & \nabla g_\rho(\bz_\rho \mid \bz_\rho)^t(\bx-\bz_\rho)+ \frac{L+\rho\|\bD\|^2}{2}\|\bx-\bz_\rho\|^2 
\label{smoothineq} \\
& = & \frac{L+\rho \|\bD\|^2}{2}\|\bx-\bz_\rho\|^2 . \nonumber
\end{eqnarray}
The strong convexity condition
\begin{eqnarray*}
0 & \ge & h_\rho(\bz_\rho) - h_\rho (\bx) \amp \ge \amp \nabla h_\rho (\bx)^t
(\bz_\rho-\bx) +\frac{\mu}{2}\|\bz_\rho-\bx\|^2
\end{eqnarray*}
entails
\begin{eqnarray*}
\|\nabla h_\rho(\bx) \|\cdot \|\bz_\rho-\bx\| & \ge & 
-\nabla h_\rho(\bx)^t(\bz_\rho-\bx) \amp \ge \amp \frac{\mu}{2}\|\bz_\rho-\bx\|^2 .
\end{eqnarray*}
It follows that $\|\nabla h_\rho(\bx) \| \ge \frac{\mu}{2}\|\bx-\bz_\rho\|$.
This last inequality and inequality (\ref{smoothineq}) produce the 
Polyak-{\L}ojasiewicz bound
\begin{eqnarray*}
\frac{1}{2}\|\nabla h_\rho(\bx) \|^2 & \ge & \frac{\mu^2}{2(L+\rho\|\bD\|^2)}
[h_\rho(\bx)-h_\rho(\bz_\rho)].
\end{eqnarray*}
Taking $c=L+\rho\|\bD\|^2$ and 
\begin{eqnarray*}
\bx = \bx_n- c^{-1} \nabla g_\rho(\bx_n \mid \bx_n) \amp = \amp \bx_n- c^{-1} \nabla h_\rho(\bx_n),
\end{eqnarray*}
the Polyak-{\L}ojasiewicz bound gives
\begin{eqnarray*}
h_\rho(\bx_{n+1})-h_\rho(\bx_n) & \le &
g_\rho(\bx_{n+1} \mid \bx_n)-g_\rho(\bx_n \mid \bx_n) \\
& \le & g_\rho(\bx \mid \bx_n)-g_\rho(\bx_n \mid \bx_n) \\
& \le & - c^{-1} \nabla g_\rho(\bx_n \mid \bx_n)^t\nabla h_\rho(\bx_n)
+\frac{c}{2}\| c^{-1}\nabla h_\rho(\bx_n)\|^2 \\
& = & - \frac{1}{2c}\| \nabla h_\rho(\bx_n)\|^2 \\
& \le & - \frac{\mu^2}{2c^2}
[h_\rho(\bx_n)-h_\rho(\bz_\rho)].
\end{eqnarray*}
Rearranging this inequality yields
\begin{eqnarray*}
h_\rho(\bx_{n+1})-h_\rho(\bz_\rho) & \le &
\Big[1- \frac{\mu^2}{2c^2}\Big]
[h_\rho(\bx_n)-h_\rho(\bz_\rho))],
\end{eqnarray*}
which can be iterated to give the stated bound.
\end{proof}

\subsection{Proposition \ref{order_stats_prop}}

\begin{proof} The first claim is true owing to the inclusion-exclusion formula 
\begin{eqnarray*}
f_{(k)}(\bx) & = & \sum_{j=k}^n \sum_{|S|=j}(-1)^{j-k}\binom{j-1}{k-1}\max\{f_i(\bx), i \in S\}
\end{eqnarray*}
and the previously stated closure properties. For $n=3$ and $k=2$ the inclusion-exclusion formula reads $f_{(2)}=\max\{f_1,f_2\}+\max\{f_1,f_3\}+\max\{f_2,f_3\}-2\max\{f_1,f_2,f_3\}$. To prove the second claim, note that a sparsity set in $\mathbb{R}^p$ with at most $k$ nontrivial coordinates can be expressed as the zero set $\{\bx: y_{(p-k)}=0\}$, where $y_i=|x_i|$. Thus, it is semialgebraic. 
\end{proof}

\subsection{Proposition \ref{non-convex-convergence}}
\begin{proof}
To validate the subanalytic premise of Proposition
\ref{MMconvergence}, first note that semialgebraic functions and sets are automatically subanalytic. The penalized loss
\begin{eqnarray*}
h_\rho(\bx) & = & f(\bx)+\frac{\rho}{2}\dist(\bx, S)^2
\end{eqnarray*}
is semialgebraic by the sum rule. Under the assumption stated in Proposition \ref{propositionb}, $g_\rho(\bx \mid \bx_n)$ is strongly convex and coercive for $\rho$ sufficiently large. Continuity of $g_\rho(\bx \mid \bx_n)$ is a consequence of the continuity of $f(\bx)$. The Lipschitz condition follows from the fact that the sum of two Lipschitz functions is Lipschitz. Under these conditions and regardless of which projected point $P_{S}(\bx)$ is chosen, the MM iterates are guaranteed to converge to a stationary point.
\end{proof}

\subsection{Proposition \ref{sparsity-convergence}}
\begin{proof}
Proposition \ref{non-convex-convergence} proves that the proximal distance iterates $\bx_n$ converge to $\bx_\infty$. Suppose that $\bD\bx_\infty$ has $k$ unambiguous largest components in magnitude. Then $\bD\bx_n$ shares this property for large $n$. It follows that all $\bp_n=\mathcal{P}_{S_k}(\bD\bx_n)$ occur in the same $k$-dimensional subspace $S$ for large $n$. Thus, we can replace the sparsity set $S_k$ by the subspace $S$ in minimization from some $n$ onward. Convergence at a linear rate now follows from Proposition \ref{proposition0}.

To prove that the set $A$ of points $\bx$ such that $\bD\bx$ has $k$ ambiguous largest components in magnitude has measure $0$, observe that it is contained in the set $T$ where two or more coordinates tie. Suppose $\bx$ satisfies the tie condition $\bd_i^t\bx=\bd_j^t\bx$ for two rows $\bd_i^t$ and $\bd_j^t$ of $\bD$. If the rows of $\bD$ are unique, then the equality $(\bd_i-\bd_j)^t\bx=0$ defines a hyperplane in $\bx$ space and consequently has measure $0$. Because there are a finite number of row pairs, $T$ as a union has measure $0$.
\end{proof}

\begin{appendices}

\appendix

\section{Convergence Properties of ADMM}
\label{appendix:ADMM}

To avail ourselves of the known results, we define three functions  
\begin{eqnarray*}
H_\rho(\bx,\by)& = & f(\bx)+\frac{\rho}{2}\dist(\by,S)^2 \\
\mathcal{L}_\rho(\bx, \by, \blambda) & = & H_\rho(\bx,\by) + \lambda^t(\bD\bx-\by)\\
q(\blambda) & = & \min_{\bx,\by} \mathcal{L}_\rho(\bx, \by, \blambda),
\end{eqnarray*}
the second and third being the Lagrangian and dual function. This notation leads to following result; see \cite{beck2017first} for an accessible proof.
\begin{secprop} \label{admm_convergence}
Suppose that $S$ is closed and convex and that the loss $f(\bx)$ is  proper, closed, and convex with
domain whose relative interior is nonempty.  Also assume the dual function $q(\blambda)$ achieves its maximum value. If the objective $f(\bx)+ \frac{\rho}{2}\|\bD\bx\|^2 + \ba^t \bx$ achieves its minimum value for all $\ba \ne {\bf 0}$, then the ADMM running averages 
\begin{eqnarray*}
\bar{\bx}_n & = & \frac{1}{n}\sum_{k=1}^n \bx_k \;\; \text{and} \;\; 
\bar{\by}_n \amp = \amp \frac{1}{n}\sum_{k=1}^n \by_k 
\end{eqnarray*}
satisfy
\begin{eqnarray*}
\vert H_\rho(\bar{\bx}_n,\bar{\by}_n)-h_\rho(\bx_\rho)\vert  &  = &  O\Big(\frac{\rho}{n} \Big) \\
\| D \bar{\bx}_n -\by_n\| & =  &  O\Big(\frac{1}{k} \Big).
\end{eqnarray*}
\end{secprop}

Note that Proposition \ref{propositionb} furnishes a sufficient condition under which the functions $f(\bx)+ \frac{\rho}{2}\|\bD\bx\|^2 + \ba^t \bx$ achieve their minima. Linear convergence holds under stronger assumptions.

\begin{secprop} \label{admm_linear_convergence}
Suppose that $S$ is closed and convex, that the loss $f(\bx)$ is $L$-smooth and $\mu$-strongly convex,
and that the map determined by $\bD$ is onto. Then the ADMM iterates converge at a linear rate.
\end{secprop}

\cite{giselsson2016linear} proved Proposition \ref{admm_linear_convergence} by operator methods. A range of convergence rates is 
specified there.

\section{Additional Details for Metric Projection Example}
\label{appendix:example1}

Given a \(n \times n\) dissimilarity matrix $\bC = (c_{ij})$ with non-negative weights $w_{ij}$, our goal is to find a semi-metric \(\bX = (x_{ij})\).
We start by denoting  $\mathrm{trivec}$ an operation that maps a symmetric  matrix $\bX$ to a vector \(\bx\), \(\bx = \mathrm{trivec}(\bX)\) (Figure \ref{fig:trivec-operation}). Then we write the metric projection objective as
\begin{equation*}
h_{\rho}(\bx)
=
\frac{1}{2} \|\bW^{1/2}(\bx - \bc)\|^{2}_{2}
+ \frac{\rho}{2} \dist(\bT \bx, \mathbb{R}_{+}^{m_{1}})^2
+ \frac{\rho}{2} \dist(\bx, \mathbb{R}_{+}^{m_{2}})^2,
\end{equation*}
where   \(\bc = \mathrm{trivec}(\bC)\).
\begin{figure}[!htbp]
	\centering
	\includegraphics[width=0.5\textwidth]{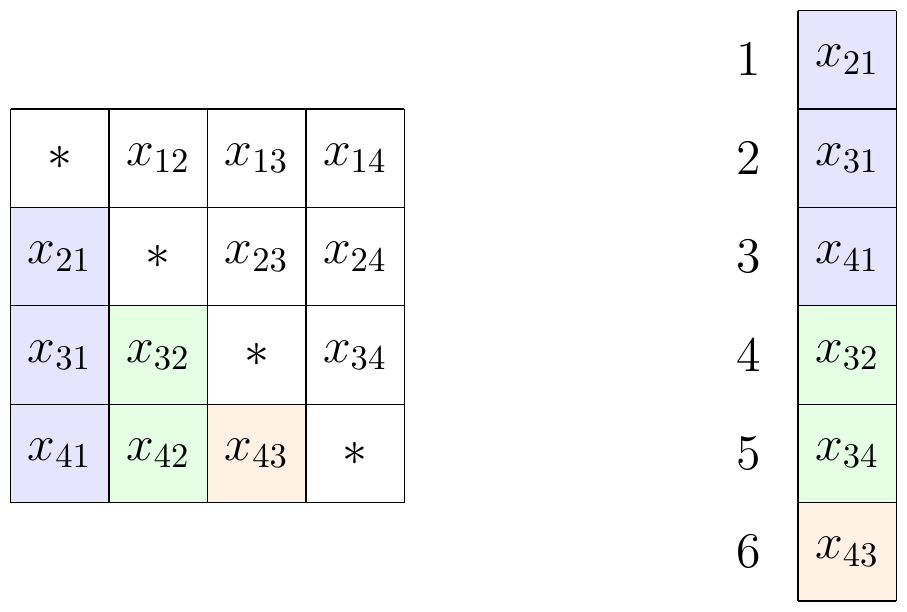}
	\caption{
		Example of a symmetric matrix \(\bX\) and its minimal representation \(\bx = \mathrm{trivec}(\bX)\).
	}
	\label{fig:trivec-operation}
\end{figure}
Here $\bT$ encodes triangle inequalities and the $m_{i}$ count the number of contraints of each type.
The usual distance majorization furnishes a surrogate
\begin{align*}
g_{\rho}(\bx \mid \bx_{n})
&=
\frac{1}{2} \|\bW^{1/2} (\bx - \bc)\|^{2}_{2}
+
\frac{\rho}{2} \|\bT \bx - \mathcal{P}(\bT \bx_{n}, \mathbb{R}_{+}^{m_{1}})\|^{2}_{2}
+
\frac{\rho}{2} \lVert{\bx - \mathcal{P}(\bx_{n}, \mathbb{R}_{+}^{m_{2}})}\rVert^{2}_{2} \\
&=
\frac{1}{2} \|\bW^{1/2} (\bx - \bc)\|^{2}_{2}
+
\frac{\rho}{2} \|\bD \bx - \mathcal{P}(\bD \bx_{n})\|_{2}^{2}.
\end{align*}
The notation \(\mathcal{P}(\cdot, S)\) denotes projection onto a set \(S\).
The fusion matrix $\bD = [\bT; \bI]$ stacks the two operators; the joint projection operates in a block-wise fashion.

\subsection{MM}

We rewrite the surrogate explicitly as a least squares problem minimizing $\|\bA \bx - \bb_{n}\|^{2}_{2}$:
\begin{equation*}
\bx_{n+1} = \underset{\bx}{\argmin} \frac{1}{2} \left\|
\begin{bmatrix}
\bW^{1/2} \\
\sqrt{\rho} \bD
\end{bmatrix} \bx
-
\begin{bmatrix}
\bc \\
\sqrt{\rho} \mathcal{P}(\bD \bx_{n})
\end{bmatrix}
\right\|_{2}^{2},
\end{equation*}
where $\bc \equiv \by$ from the main text.
Updating the RHS $\bb_{n}$ in the linear system reduces to evaluating the projection and copy operations.
It is worth noting that triangle fixing algorithms that solve the metric nearness problem operate in the same fashion, except they work one triangle a time.
That is, each iteration solves $\binom{n}{3}$ least squares problems compared to 1 in this formulation.
A conjugate gradient type of algorithm solves the normal equations directly using $\bA^{t}\bA$, whereas LSQR type methods use only $\bA$ and $\bA^{t}$.

\subsection{Steepest Descent}

The updates $\bx_{n+1} = \bx_{n} - t_{n} \nabla h_{\rho}(\bx_{n})$ admit an exact solution for the line search parameter $t_{n}$.
Recall the generic formula from the main text:
\begin{equation*}
t_{n} = \frac{\|\bv_{n}\|^{2}}{\bv_{n}^{t}\bA\bv_{n} + \rho \|\bD \bv_{n}\|^{2}}.
\end{equation*}
Identifying $\bv_{n}$ with $\nabla h_{\rho}(\bx_{n})$ we have
\begin{align*}
\nabla h_{\rho}(\bx_{n})
&= \bW(\bx_{n} - \bc) + \rho \bD^{t} [\bD \bx_{n} - \mathcal{P}(\bD \bx_{n})] \\
&= \bW(\bx_{n} - \bc) + \rho (\bI + \bT^{t}\bT) \bx_{n} - \rho[\bT^{t} \mathcal{P}(\bT \bx_{n}, \mathbb{R}_{+}^{m_{1}}) + \mathcal{P}(\bx_{n}, \mathbb{R}_{+}^{m_{1}})], \\
t_{n}
&=
\frac{\|\bv_{n}\|^{2}}{\|\bW^{1/2} \bv_{n}\|^{2} + \rho \|\bD \bv_{n}\|^{2}}.
\end{align*}

\subsection{ADMM}

Taking $\by$ as the dual variable and $\blambda$ as scaled multipliers, the updates for each ADMM block are

\begin{align*}
\bx_{n+1}
&= \underset{\bx}{\argmin} \left\|
\begin{bmatrix}
\bW^{1/2} \\
\sqrt{\mu} \bD
\end{bmatrix} \bx
-
\begin{bmatrix}
\bc \\
\sqrt{\mu} (\by_{n} - \blambda_{n}))
\end{bmatrix}
\right\|_{2}^{2}, \\
\by_{n+1}
&= \frac{\alpha}{1+\alpha} \mathcal{P}(\bz_{n}) + \frac{1}{1+\alpha} \bz_{n};
\qquad \bz_{n} = \bD \bx_{n+1} + \blambda_{n},~\alpha = \rho / \mu.
\end{align*}
Finally, the Multipliers follow the standard update.

\subsection{ Properties of the Triangle Inequality Matrix}

These results have been documented before and are useful in designing fast subroutines for $\bD \bx$ and $\bD^t \bD \bx$. Recall that $m$ counts the number of nodes in the problem and $p = \binom{m}{2}$ is the number of parameters. In this notation $\bD = \begin{pmatrix} \bT \\ \bI_p \end{pmatrix}$ and 
$\bD^t\bD= \bT^t\bT + \bI_p$.

\begin{proposition}
	The matrix \(\bT\) has \(3 \binom{m}{3}\) rows and \(\binom{m}{2}\) columns.
\end{proposition}
\begin{proof}
	Interpret \(\bX\) as the adjacency matrix for a complete directed graph on \(m\) nodes without self-edges.
	When \(\bX\) is symmetric the number of free parameters is therefore \(\binom{m}{2}\).
	An oriented \(3\)-cycle is formed by fixing \(3\) nodes so there are \(\binom{m}{3}\) such cycles.
	Now fix the orientation of the \(3\)-cycles and note that each triangle encodes \(3\) metric constraints.
	The number of constraints is therefore \(3 \binom{m}{3}\).
\end{proof}

\begin{proposition}
	Each column of \(\bT\) has \(3 (m-2)\) nonzero entries.
\end{proposition}
\begin{proof}
	In view of the previous result, the entries \(T_{ij}\) encode whether edge \(j\) participates in constraint \(i\).
	We proceed by induction on the number of nodes \(m\).
	The base case \(m = 3\) involves one triangle and is trivial.
	Note that a triangle encodes \(3\) inequalities.
	
	Now consider a complete graph on \(m\) nodes and suppose the claim holds.
	Without loss of generality, consider the collection of \(3\)-cycles oriented clockwise and fix an edge \(j\).
	Adding a node to the graph yields \(2 m\) new edges, two for each of the existing \(m\) nodes.
	This action also creates one new triangle for each existing edge.
	Thus, edge \(j\) appears in \(3(m-2) + 3 = 3 [(m+1)-2]\) triangle inequality constraints based on the induction hypothesis.
\end{proof}

\begin{proposition}
	Each column of \(\bT\) has \(m-2\) \(+1\)s and \(2(m-2)\) \(-1\)s.
\end{proposition}
\begin{proof}
	Interpret the inequality \(x_{ij} \le x_{ik} + x_{kj}\) with \(i > k > j\) as the ordered triple \(x_{ij}, x_{ik}, x_{kj}\).
	The statement is equivalent to counting
	\begin{align*}
	a(N)
	&=
	\text{number of times \(x_{ij}\) appears in position 1, and},\\
	b(N)
	&=
	\text{number of times \(x_{ij}\) appears in position 2 or 3},
	\end{align*}
	where \(N\) denotes the number of constraints.
	In view of the previous proposition, it is enough to prove \(a(N) = m-2\).
	Note that \(a(3) = 1\), meaning that \(x_{ij}\) appears in position \(1\) exactly once within a given triangle.
	Given that an edge \((i,j)\) appears in \(3 (m-2)\) constraints, divide this quantity by the number of constraints per triangle to arrive at the stated result.
\end{proof}

\begin{proposition}
	The matrix \(\bT\) has full column rank.
\end{proposition}
\begin{proof}
	It is enough to show that \(\bA = \bT^t\bT\) is full rank.
	The first two propositions imply
	\begin{equation*}
	a_{ii}
	=
	\langle{\bT_{i}, \bT_{i}}\rangle
	=
	\sum (\pm 1)^{2}
	=
	3 (m-2).
	\end{equation*}
	To compute the off-diagonal entries, fix a triangle and note that two edges \(i\) and \(j\) appear in all three of its constraints of the form \(x_{i} \le x_{j} + x_{k}\).
	There are three possibilities for a given constraint \(c\):
	\begin{equation*}
	T_{c,i} T_{c,j} \amp = \amp
	\begin{cases}
	-1, &\text{if \(i\) LHS, \(j\) RHS or vice-versa} \\
	\amp 1,  &\text{if \(i\) and \(j\) both appear on RHS} \\
	\amp 0,  &\text{if one of \(i\) or \(j\) is missing}.
	\end{cases}
	\end{equation*}
	It follows that
	\begin{equation*}
	a_{ij}
	=
	\langle{\bT_{i}, \bT_{j}}\rangle
	=
	\begin{cases}
	-1, &\text{if edges $i$ and $j$ overlap in constraints} \\
	\amp 0, &\text{otherwise}.
	\end{cases}
	\end{equation*}
	By Proposition B.2, an edge \(i\) appears in \(3 (m-2)\) constraints.
	Imposing the condition that edge \(j\) also appears reduces this number by \(m-2\), the number of remaining nodes that can contribute edges in our accounting.
	The calculation
	\begin{equation*}
	\sum_{j \neq i} |a_{ij}| = 2 (m-2) < 3 (m-2) = |a_{ii}|,
	\end{equation*}
	establishes that \(\bA\) is strictly diagonally dominant and hence full rank.
\end{proof}

\begin{proposition}
	The matrix \(\bT^t\bT\) has at most \(3\) distinct eigenvalues of the form \(m-2\), \(2m-2\), and \(3m-4\) with multiplicities \(1\), \(m-1\), and \(\frac{1}{2} m (m-3)\), respectively.
\end{proposition}
\begin{proof}
Let $\bM \in \{0,1\}^{\binom{m}{2} \times m}$ be the incidence matrix of a complete graph with $m$ vertices. That is $\bM$ has entry $m_{e,v}=1$ if vertex $v$ occurs in edge $e$ and 0 otherwise. Each row of $\bM$ has two entries equal to 1; each column of $\bM$ has $m-1$ entries equal to 1. It is easy to see
\[
\bT^t \bT = (3m-4) \bI_{\binom{m}{2}} - \bM \bM^t.
\]
The Gram matrices $\bM^t \bM$ and $\bM \bM^t$ share the same positive eigenvalues. Since $\bM^t \bM = (m-2) \bI_m + m (\boldsymbol{1}_m / \sqrt m) (\boldsymbol{1}_m / \sqrt m)^t$ has eigenvalue $2m-2$ with multiplicity 1 and eigenvalue $m-2$ with multiplicity $m-1$, $\bM \bM^t$ has eigenvalue $2m-2$ with multiplicity 1, eigenvalue $m-2$ with multiplicity $m-1$, and eigenvalue 0 with multiplicity $m(m-3)/2$. Therefore the eigenvalues of $\bT^t \bT$ are $m-2$, $2m-2$, and $3m-4$ with multiplicities $1$, $m-1$, and $m(m-3)/2$ respectively.
\end{proof}

In general, it is easy to check that the matrix $m \times m$ matrix $a \bI + b{\bf 1}{\bf 1}^t$ has the eigenvector ${\bf 1}$ with eigenvalue $a+mb$ and $m-1$ orthogonal eigenvectors 
\begin{eqnarray*}
\bu_i & = & \frac{1}{i-1} \sum_{j=1}^{i-1} {\bf e}_j-{\bf e}_i ,
\qquad i=2,\ldots,m
\end{eqnarray*}
with eigenvalue $a$.  Note that each $\bu_i$ is perpendicular to 
${\bf 1}$. None of these eigenvectors is normalized to have length 1. Although the
eigenvectors $\bu_i$ are certainly convenient, they are not unique.

To recover the eigenvectors of $\bT^t\bT$, and hence those $\bD^t\bD$, we can leverage the eigenvectors of $\bM^t \bM$, which we know. The following generic observations are pertinent.
If a matrix $\bA$ has full SVD $\bU\bS\bV^t$, then its transpose has full SVD 
$\bA^t = \bV\bS\bU^t$. As mentioned $\bA\bA^t$ and $\bA^t\bA$ share the same nontrivial
eigenvalues. These can be recovered as the nontrivial diagonal entries of $\bS^2$. 
Suppose we know the eigenvectors $\bU$ of $\bA\bA^t = \bU\bS^2\bU^t$. Since  
$\bA^t \bU = \bV\bS$, then presumably we can recover some of the eigenvectors 
$\bV$ as $\bA^t\bU\bS^{+}$, where $\bS^+$ is the diagonal pseudo-inverse of $\bS$.

\subsection{Fast Subroutines for Solving Linear Systems}
\label{appendix:example1-update}
Using the Woodbury formula, the inverse of $\bT^t \bT$ can be expressed as
\begin{eqnarray*}
& & (\bT^t \bT)^{-1} \\
&=& \left[ (3m-4) \bI_{\binom{m}{2}} - \bM \bM^t \right]^{-1} \\
&=& (3m-4)^{-1} \bI_{\binom{m}{2}} - (3m-4)^{-2} \bM [-\bI_m + (3m-4)^{-1} \bM^t \bM]^{-1} \bM^t \\
&=& (3m-4)^{-1} \bI_{\binom{m}{2}} - (3m-4)^{-1} \bM [-(2m-2)\bI_m + \boldsymbol{1}_m \boldsymbol{1}_m^t]^{-1} \bM^t \\
&=& (3m-4)^{-1} \bI_{\binom{m}{2}} - (3m-4)^{-1} \bM [-(2m-2)^{-1} \bI_m - (2m-2)^{-1}(m-2)^{-1} \boldsymbol{1}_m \boldsymbol{1}_m^t] \bM^t \\
&=& \frac{1}{3m-4} \bI_{\binom{m}{2}} + \frac{2}{(3m-4)(m-1)(m-2)} \boldsymbol{1}_{\binom{m}{2}} \boldsymbol{1}_{\binom{m}{2}}^t + \frac{1}{2(3m-4)(m-1)} \bM \bM^t.
\end{eqnarray*}
Solving linear system $\bT^t \bT$ invokes two matrix vector multiplications involving the incidence matrix $\bM$. $\bM \bv$ corresponds to taking pairwise sums of the components of a vector $\bv$ of length $m$. $\bM^t \bw$ corresponds to taking a combination of column and row sums of a lower triangular matrix with the lower triangular part populated by the components of a vector $\bw$ with length $\binom{m}{2}$. Both operations cost $O(m^2)$ flops.
This result can be extended to the full fusion matrix $\bD^{t} \bD$ that incorporates non-negativity constraints and, more importantly, to the linear system $\bI + \rho \bD^{t} \bD$:
\begin{equation*}
\begin{split}
	[\bI_{\binom{m}{2}} + \rho \bD^{t} \bD]^{-1}
	&=
	[\bI_{\binom{m}{2}} + \rho (\bT^{t}\bT + \bI_{\binom{m}{2}})]^{-1} \\
	&=
	a~\bI_{\binom{m}{2}}
	+ ab\rho~\bM \bM^{t}
	+ 4abc\rho^{2}~\boldsymbol{1}_{\binom{m}{2}} \boldsymbol{1}_{\binom{m}{2}}^{t}; \\
\end{split}
\qquad
\begin{split}
	a &= [3(m-1)\rho + 1]^{-1} \\
	b &= [(2m-1)\rho + 1]^{-1} \\
	c &= [(m-1)\rho + 1]^{-1}.
\end{split}
\end{equation*}

\section{Additional Details for Convex Regression Example}
\label{appendix:example2}

We start by formulating the proximal distance version of the problem:
\begin{equation*}
h_{\rho}(\bv)
=
\frac{1}{2} \|\bM \bv - \by\|^{2}_{2}
+ \frac{\rho}{2} \dist(\bD \bv, \mathbb{R}_{-}^{m})^{2},
\end{equation*}
where $\bv = [\btheta; \vec(\bXi)]$ stacks each optimization variable into a vector of length $n(1+d)$.
This maneuver introduces matrices
\begin{equation*}
\bM = \begin{bmatrix}
\bI_{n \times n} & \boldsymbol{0}_{n \times nd}
\end{bmatrix},
\qquad
\bD = \begin{bmatrix}
\bA & \bB
\end{bmatrix},
\end{equation*}
where $[\bA \btheta]_{k} = \theta_{j} - \theta_{i}$ and $[\bB \vec(\Xi)]_{k} = \langle{\bx_{i} - \bx_{j}, \bxi_{j}}\rangle$ according to the ordering $i > j$.

\subsection{MM}
We rewrite the surrogate explicitly a least squares problem minimizing $\|\tilde{\bM} \bv - \tilde{\bb}_{n}\|^{2}_{2}$:
\begin{equation*}
\bv_{n+1} = \underset{\bv}{\argmin} \frac{1}{2} \left\|
\begin{bmatrix}
\bM \\
\sqrt{\rho} \bD
\end{bmatrix} \bv
-
\begin{bmatrix}
\bb \\
\sqrt{\rho} \mathcal{P}(\bD \bv_{n})
\end{bmatrix}
\right\|_{2}^{2},
\end{equation*}
where $\bb \equiv \by$ to avoid clashing with notation in ADMM below.
In this case it seems  better to store $\bD$ explicitly in order to avoid computing $\bx_{i} - \bx_{j}$ each time one applies $\bD$, $\bD^{t}$, or $\bD^{t} \bD$.

\subsection{Steepest Descent}

The updates $\bv_{n+1} = \bv_{n} - t_{n} \nabla h_{\rho}(\bv_{n})$ admit an exact solution for the line search parameter $t_{n}$.
Taking $\bq_{n} = \nabla h_{\rho}(\bv_{n})$ as the gradient we have
\begin{align*}
\bq_{n}
&= \bA^{t} \bA(\bv_{n} - \bb) + \rho \bD^{t} [\bD \bv_{n} - \mathcal{P}(\bD \bv_{n})], \\
t_{n}
&=
\frac{\|\bq_{n}\|^{2}}{\|\bA \bq_{n}\|^{2} + \rho \|\bD \bq_{n}\|^{2}}.
\end{align*}
Note that $\bA \bq_{n} = \nabla_{\btheta} h_{\rho}(\bv_{n})$, the gradient with respect to function values $\btheta$.

\subsection{ADMM}

Take $\by$ as the dual variable and $\blambda$ as scaled multipliers. Then the ADMM updates are
\begin{align*}
\bv_{n+1}
&=
\underset{\bv}{\argmin} \frac{1}{2} \left\|
\begin{bmatrix}
\bA \\
\sqrt{\mu} \bD
\end{bmatrix} \bv
-
\begin{bmatrix}
\bb \\
\sqrt{\mu} (\by_{n} - \blambda_{n})
\end{bmatrix}
\right\|_{2}^{2},\\
\by_{n+1}
&= \frac{\alpha}{1+\alpha} \mathcal{P}(\bz_{n}) + \frac{1}{1+\alpha} \bz_{n};
\qquad \bz_{n} = \bD \bv_{n+1} + \blambda_{n},~\alpha = \rho / \mu,
\end{align*}
and with the update for the multipliers being standard.

\section{Additional Details for Convex Clustering Example}
\label{appendix:example3}

We write  $\bu = \vec(\bU)$ and $\bx = \vec(\bX)$, so the surrogate becomes
\begin{equation*}
g_{\rho}(\bu \mid \bu_{n})
=
\frac{1}{2}\|\bu - \bx\|_{2}^{2}
+
\frac{\rho}{2} \|\bD \bu - \mathcal{P}_{S_{k}}(\bD \bu_{n})\|^{2}.
\end{equation*}

\subsection{MM}
Rewrite the surrogate explicitly a least squares problem minimizing $\|\bA \bu - \bb_{n}\|^{2}_{2}$:
\begin{equation*}
\bu_{n+1} = \underset{\bu}{\argmin} \frac{1}{2} \left\|
\begin{bmatrix}
\bI \\
\sqrt{\rho} \bD
\end{bmatrix} \bu
-
\begin{bmatrix}
\bx \\
\sqrt{\rho} \mathcal{P}_{S_{k}}(\bD \bu_{n})
\end{bmatrix}
\right\|_{2}^{2}.
\end{equation*}

\subsection{Steepest Descent}

The updates $\bu_{n+1} = \bu_{n} - t_{n} \nabla h_{\rho}(\bu_{n})$ admit an exact solution for the line search parameter $t_{n}$.
Taking $\bq_{n} = \nabla h_{\rho}(\bu_{n})$ as the gradient we have
\begin{align*}
\bq_{n}
&= (\bu_{n} - \bx) + \rho \bD^{t} [\bD \bu_{n} - \mathcal{P}_{S_{k}}(\bD \bu_{n})], \\
t_{n}
&=
\frac{\|\bq_{n}\|^{2}}{\|\bq_{n}\|^{2} + \rho \|\bD \bq_{n}\|^{2}}.
\end{align*}
Note that blocks in $[\bD \bu_{n} - \mathcal{P}_{S_{k}}(\bD\bu_{n})]_{\ell}$ are equal to $\boldsymbol{0}$ whenever the projection of block $[\bD \bu_{n}]_{\ell}$ is non-zero.

\subsection{ADMM}

Take $\by$ as the dual variable and $\blambda$ as scaled multipliers.
Minimizing the $\bu$ block involves solving a single linear system:
\begin{align*}
\bu_{n+1}
&=
\underset{\bu}{\argmin} \frac{1}{2} \left\|
\begin{bmatrix}
\bI \\
\sqrt{\mu} \bD
\end{bmatrix} \bu
-
\begin{bmatrix}
\bx \\
\sqrt{\mu} (\by_{n} - \blambda_{n})
\end{bmatrix}
\right\|_{2}^{2}, \\
\by_{n+1}
&= \frac{\alpha}{1+\alpha} \mathcal{P}_{S_{k}}(\bz_{n}) + \frac{1}{1+\alpha} \bz_{n};
\qquad \bz_{n} = \bD \bu_{n+1} + \blambda_{n},~\alpha = \rho / \mu.
\end{align*}
Multipliers follow the standard update.

\subsection{Blockwise Sparse Projection}

The projection $\mathcal{P}_{S_{k}}$ maps a matrix to a sparse representation with $k$ non-zero columns (or blocks in the case of the vectorized version).
In the context of clustering, imposing sparsity permits a maximum of $k$ violations in consensus of centroid assignments, $\bu_{i} = \bu_{j}$.
Letting $\Delta_{\ell} \equiv \Delta_{ij} = \|\bu_{i} - \bu_{j}\|$ denote pairwise distances and $M = \binom{m}{2}$ denote the number of unique pairwise distances, we define the projection along blocks $\bv_{\ell} = \bu_{i} - \bu_{j}$ for each pair as
\begin{align*}
\mathcal{P}_{S_{k}}(\bv_{\ell})
=
\begin{cases}
\bv_{\ell}, & \text{if}~\Delta_{\ell} \in \{\Delta_{(M)}, \Delta_{(M-1)}, \ldots \Delta_{(M-k+1)}\} \\
\boldsymbol{0}, & \text{otherwise}.
\end{cases}
\end{align*}
Here the notation $x_{(i)}$ represents the $i$-th element in an ascending list.
Concretely, the magnitude of a difference $\bv_{\ell}$ must be within the top $k$ distances.
An alternative, helpful definition is based on the smallest distances
\begin{align*}
\mathcal{P}_{S_{k}}(\bv_{\ell})
=
\begin{cases}
\boldsymbol{0}, & \text{if}~\Delta_{k} \in \{\Delta_{(1)}, \Delta_{(2)}, \ldots \Delta_{(k)}\} \\
\bv_{\ell}, & \text{otherwise}
\end{cases}
\end{align*}
Thus, it is enough to find a pivot $\Delta_{(M-k+1)}$ or $\Delta_{(k)}$ that splits the list into the top $k$ elements.
Because the hyperparameter $k$ has a finite range in $\{0,1,2,\ldots,\binom{m}{2}\}$ one can exploit symmetry to reduce the best/average computational complexity in a search procedure.
We implement this projection using a partial sorting algorithm based on quicksort, and note that it is set-valued in general.

\section{Additional Details for Image Denoising  Example}
\label{appendix:example4}

Here we restate the total variation denoising problem to take advantage of proximal operators in the proximal distance framework.
We minimize the penalized objective
\begin{align*}
h_{\rho}(\bU)
&=
\frac{1}{2}\|\bu - \bw\|_{F}^{2}
+
\frac{\rho}{2} \dist(\bD \bu, S_{\gamma})^{2},
\end{align*}
where $\bw = \vec(\bW)$ is a noisy input image and $S_{\gamma}$ is the $\ell_{1}$ ball with radius $\gamma$. Thus, $\gamma$ may be interpreted as the target total variation of the reconstructed image.
Distance majorization yields the surrogate
\begin{equation*}
g_{\rho}(\bu \mid \bu_{n})
=
\frac{1}{2}\|\bu - \bw\|_{2}^{2}
+
\frac{\rho}{2} \|\bD \bu - \mathcal{P}_{\gamma}(\bD \bu_{n})\|^{2}.
\end{equation*}
Here $\mathcal{P}_{\gamma}(\bD \bu)$ enforces sparsity in all derivatives through projection onto the $\ell_{1}$ ball with radius $\gamma$.
Because $\bD$ is ill-conditioned, we append an additional row with zeros everywhere except the last entry; that is, $\bD = [\bD_{n}, \bD_{p}, \be_{p}]$ with $\bu \in \mathbb{R}^{p}$.
In this case, the sparse projection applies to all but the last component of $\bD \bu$.

\subsection{MM}
Rewrite the surrogate explicitly as a least squares problem:
\begin{equation*}
\bu_{n+1} = \underset{\bx}{\argmin} \frac{1}{2} \left\|
\begin{bmatrix}
\bI \\
\sqrt{\rho} \bD
\end{bmatrix} \bu
-
\begin{bmatrix}
\bw \\
\sqrt{\rho} \mathcal{P}_{\gamma}(\bD \bu_{n})
\end{bmatrix}
\right\|_{2}^{2}.
\end{equation*}

\subsection{Steepest Descent}

The updates $\bu_{n+1} = \bu_{n} - t_{n} \nabla h_{\rho}(\bu_{n})$ admit an exact solution for the line search parameter $t_{n}$.
Taking $\bq_{n} = \nabla h_{\rho}(\bu_{n})$ as the gradient we have
\begin{align*}
\bq_{n}
&= (\bu_{n} - \bw) + \rho \bD^{t} [\bD \bu_{n} - \mathcal{P}_{\gamma}(\bD \bu_{n})], \\
t_{n}
&=
\frac{\|\bq_{n}\|^{2}}{\|\bq_{n}\|^{2} + \rho \|\bD \bq_{n}\|^{2}}.
\end{align*}

\subsection{ADMM}

We denote by  $\by$  the dual variable and $\blambda$ the scaled multipliers.
Minimizing the $\bu$ block involves solving a single linear system:
\begin{align*}
\bu_{n+1}
&=
\underset{\bx}{\argmin} \frac{1}{2} \left\|
\begin{bmatrix}
\bI \\
\sqrt{\mu} \bD
\end{bmatrix} \bu
-
\begin{bmatrix}
\bw \\
\sqrt{\mu} (\by_{n} - \blambda_{n})
\end{bmatrix}
\right\|_{2}^{2}, \\
\by_{n+1}
&= \frac{\alpha}{1+\alpha} \mathcal{P}_{\gamma}(\bz_{n}) + \frac{1}{1+\alpha} \bz_{n};
\qquad \bz_{n} = \bD \bu_{n+1} + \blambda_{n},~\alpha = \rho / \mu.
\end{align*}
Multipliers follow the standard update.

\section{Additional Details for Condition Number Example}
\label{appendix:example5}

Given a matrix $\bM = \bU \bSigma \bV^{-1}$ with singular values $\sigma_{1} \ge \sigma_{2} \ge \ldots \ge \sigma_{p}$, we seek a new matrix $\bN = \bU \bX \bV^{-1}$ such that $\mathrm{cond}(\bB) = x_{1} / x_{p} \le c$.
We minimize the penalized objective
\begin{align*}
h_{\rho}(\bx)
&=
\frac{1}{2}\|\bx - \bsigma\|^{2}
+
\frac{\rho}{2} \dist(\bD \bx, \mathbb{R}^{p^{2}}_{-})^{2},
\end{align*}
as suggested by the Von Neumann-Fan inequality.
The fusion matrix $\bD = \bC + \bS$ encodes the constraints $x_{i} - c x_{j} \le 0$.
Distance majorization yields the surrogate
\begin{equation*}
g_{\rho}(\bx \mid \bx_{n})
=
\frac{1}{2}\|\bx - \bw\|_{2}^{2}
+
\frac{\rho}{2} \|\bD \bx - \mathcal{P}_{-}(\bD \bx_{n})\|^{2}.
\end{equation*}
To be specific, the matrix $\bC=-c {\bf 1}_{p} \otimes \bI_{p}$ scales the $p \times p$ identity matrix by $-c$ and stacks it $p$ times. Similarly, the matrix $\bS = \bI_{p} \otimes {\bf 1}_{p}$ stacks $p$ matrices of dimension $p \times p$. Each of these stacked matrices has $(p-1)$ ${\bf 0}_p$ columns and one shifted ${\bf 1}_p$ column. For example, for $p=2$
\begin{eqnarray*}
\bS & = & \begin{bmatrix} \begin{pmatrix} 1 & 0 \\ 1 & 0 \end{pmatrix} \\
\begin{pmatrix} 0 & 1 \\ 0 & 1 \end{pmatrix} \end{bmatrix}.
\end{eqnarray*}

\subsection{MM}
Rewrite the surrogate explicitly a least squares problem minimizing $\|\bA \bx - \bb_{n}\|^{2}_{2}$:
\begin{equation*}
\bx_{n+1} = \underset{\bx}{\argmin} \frac{1}{2} \left\|
\begin{bmatrix}
\bI \\
\sqrt{\rho} \bD
\end{bmatrix} \bx
-
\begin{bmatrix}
\bsigma \\
\sqrt{\rho} \mathcal{P}(\bD \bx_{n})
\end{bmatrix}
\right\|_{2}^{2}.
\end{equation*}
Applying the matrix inverse from before yields an explicit formula (with $a$ and $b$ defined as before):
\begin{equation*}
\bx_{n+1}
=
\frac{1}{a}\left[
\bz_{n} - \frac{\boldsymbol{1}^{t}\bz_{n}}{p - (a/b)} \boldsymbol{1}
\right];
\qquad
\bz_{n} = \bsigma + \rho \bD^{t} \mathcal{P}(\bD \bx_{n}),
~a = 1 + \rho p (c^{2} + 1), ~b = 2 \rho c.
\end{equation*}

\subsection{Steepest Descent}

The updates $\bx_{n+1} = \bx_{n} - t_{n} \nabla h_{\rho}(\bx_{n})$ admit an exact solution for the line search parameter $t_{n}$.
Taking $\bq_{n} = \nabla h_{\rho}(\bx_{n})$ as the gradient we have
\begin{align*}
\bq_{n}
&= (\bx_{n} - \bu) + \rho \bD^{t} [\bD \bx_{n} - \mathcal{P}_{\nu}(\bD \bx_{n})], \\
t_{n}
&=
\frac{\|\bq_{n}\|^{2}}{\|\bq_{n}\|^{2} + \rho \|\bD \bq_{n}\|^{2}}.
\end{align*}

\subsection{ADMM}

Take $\by$ as the dual variable and $\blambda$ as scaled multipliers.
The formula for the MM algorithm applies in updating $\bx_{n}$, except we replace $\rho$ with $\mu$ and $\mathcal{P}(\bD \bx_{n})$ with $\by_{n} - \blambda_{n}$:
\begin{align*}
\bx_{n+1}
&=
\frac{1}{a}\left[
\bz^{1}_{n} - \frac{\boldsymbol{1}^{t}\bz^{1}_{n}}{p - (a/b)} \boldsymbol{1}
\right];
\qquad
\bz_{n}^{1}
=
\bsigma + \mu \bD^{t} (\by_{n} - \blambda_{n}),
~a = 1 + \mu p (c^{2} + 1), ~b = 2 \mu c \\
\by_{n+1}
&= \frac{\alpha}{1+\alpha} \mathcal{P}(\bz^{2}_{n}) + \frac{1}{1+\alpha} \bz^{2}_{n};
\qquad \bz^{2}_{n} = \bD \bx_{n+1} + \blambda_{n},~\alpha = \rho / \mu.
\end{align*}
Multipliers follow the standard update.

\subsection{Explicit Matrix Inverse}

Both ADMM and MM reduce to solving a linear system.
Fortunately, the Hessian for $h_{\rho}(\bx)$ reduces to a Householder-like matrix.
First we note that it is trivial to multiply either $\bC^t$ or $\bS^t$ by a $p^2$-vector. The more interesting problem is calculating $(\bA_\rho^t\bA_\rho)^{-1}$, where $\nabla h_{\rho}^{2} = \bA_\rho^t\bA_\rho=\bI_p+\rho \bD^t\bD$. The reader can check the identities 
\begin{eqnarray*}
\bC^t\bC & = &  c^2p \bI_{p} \;\; \text{and} \;\;  \bC^t\bS \amp = \amp -c{\bf 1}_p {\bf 1}_p^t , \\
\bS^t\bC & = & -c {\bf 1}_p {\bf 1}_p^t \;\; \text{and} \;\; \bS^t\bS \amp = \amp p \bI_p.
\end{eqnarray*}
It follows that $\bA_\rho^t\bA_\rho = (1+\rho p (c^{2} + 1))\bI_{p} - 2c\rho {\bf 1}_{p}{\bf 1}_{p}^t$.
Applying the Sherman-Morrison formula to results in
\begin{align*}
[\bI_{p} + \rho \bD^{t} \bD]^{-1}
&=
[a \bI_{p} - b \boldsymbol{1}_{p} \boldsymbol{1}_{p}^{t}]^{-1} \\
&=
-b^{-1} \left[
-(b/a) \bI_{p} - \frac{(b/a)^{2} \boldsymbol{1}_{p} \boldsymbol{1}_{p}^{t}}{1 - (a/b) \boldsymbol{1}_{p}^{t} \boldsymbol{1}_{p}}
\right] \\
&=
\frac{1}{a} \left[
\bI_{p} - \frac{\boldsymbol{1}_{p} \boldsymbol{1}_{p}^{t}}{p - a/b}
\right],
\end{align*}
where $a = 1+\rho p (c^{2} + 1)$ and $b = 2\rho c$.
These simplifications make the exact proximal distance updates easy to compute.

\section{Choice of Linear Solver}

Updating parameters using MM or ADMM requires solving large-scale linear systems of the form $(\bI + c_{t} \bD^{t} \bD) \bx = \bb$.  
Here $c_{t}$ is a scalar that depends on the outer iteration number $t$, in general, and the matrix on the LHS is square,  symmetric, and often reasonably well-conditioned.
Standard factorization methods like Cholesky and spectral decompositions cannot be applied without efficient update rules based on $c_{t}$.
Instead, we turn to iterative methods, specifically conjugate gradients (CG) and LSQR, and use a linear map approach to adequately address sparsity, structure, and computational efficiency in matrix-vector multiplication.
Tables \ref{lsMM} and \ref{lsADMM} summarize performance metrics for MM and ADMM using both iterative linear solvers on instances of the convex regression problem.
Times are averages taken over 3 replicates with standard deviations in parentheses, and iteration counts reflect the total number of inner iterations with outer counts in parentheses.
We find no appreciable difference between CG and LSQR except on timing, and therefore favor CG in all our benchmarks.

\begin{table}[tbp]
    \centering
	\setlength{\tabcolsep}{2pt}
	\begin{scriptsize}
    \begin{tabular}{rrcccccccc}
    \toprule
    &  & \multicolumn{2}{c}{Time (s)} & \multicolumn{2}{c}{Loss $\times 10^{6}$} & \multicolumn{2}{c}{Distance $\times 10^{6}$} & \multicolumn{2}{c}{Iterations}\\
    \cmidrule(lr){3-4} \cmidrule(lr){5-6} \cmidrule(lr){7-8} \cmidrule(lr){9-10}
    features & samples & CG & LSQR & CG & LSQR & CG & LSQR & CG & LSQR\\
    \midrule
    20 & 50 & $\bf{0.39}$ & $0.68$ & $\bf{0.991}$ & $\bf{0.991}$ & $2.7$ & $\bf{2.65}$ & $\bf{104}$ & $\bf{104}$\\
    &  & $(0.00132)$ & $(0.00607)$ &  &  &  &  & $(2)$ & $(2)$\\
    20 & 100 & $\bf{1.46}$ & $2.67$ & $0.995$ & $\bf{0.994}$ & $\bf{3.08}$ & $3.16$ & $\bf{190}$ & $\bf{190}$\\
    &  & $(0.000164)$ & $(0.00555)$ &  &  &  &  & $(2)$ & $(2)$\\
    20 & 200 & $\bf{7.7}$ & $14.7$ & $\bf{0.984}$ & $\bf{0.984}$ & $\bf{25.1}$ & $\bf{25.1}$ & $\bf{298}$ & $\bf{298}$\\
    &  & $(0.166)$ & $(0.00913)$ &  &  &  &  & $(2)$ & $(2)$\\
    20 & 400 & $\bf{30.2}$ & $63$ & $\bf{0.997}$ & $\bf{0.997}$ & $9.21$ & $\bf{8.89}$ & $\bf{412}$ & $\bf{412}$\\
    &  & $(0.0275)$ & $(0.268)$ &  &  &  &  & $(2)$ & $(2)$\\
    \bottomrule
    \end{tabular}
	\end{scriptsize}
    \caption{
        \label{lsMM}
        Performance of MM on convex regression using CG and LSQR.
        Both inner and outer iterations are reported with the latter in parantheses.
    }
\end{table}

\begin{table}[tbp]
    \centering
    \setlength{\tabcolsep}{2pt}
	\begin{scriptsize}
    \begin{tabular}{rrcccccccc}
    \toprule
    &  & \multicolumn{2}{c}{Time (s)} & \multicolumn{2}{c}{Loss $\times 10^{6}$} & \multicolumn{2}{c}{Distance $\times 10^{6}$} & \multicolumn{2}{c}{Iterations}\\
    \cmidrule(lr){3-4} \cmidrule(lr){5-6} \cmidrule(lr){7-8} \cmidrule(lr){9-10}
    features & samples & CG & LSQR & CG & LSQR & CG & LSQR & CG & LSQR\\
    \midrule
    20 & 50 & $\bf{0.399}$ & $0.677$ & $\bf{0.91}$ & $\bf{0.91}$ & $\bf{6.96}$ & $\bf{6.96}$ & $\bf{98}$ & $\bf{98}$\\
    &  & $(0.00193)$ & $(0.00507)$ &  &  &  &  & $(2)$ & $(2)$\\
    20 & 100 & $\bf{1.58}$ & $2.83$ & $\bf{0.996}$ & $\bf{0.996}$ & $\bf{0}$ & $\bf{0}$ & $\bf{194}$ & $\bf{194}$\\
    &  & $(0.0329)$ & $(0.00398)$ &  &  &  &  & $(2)$ & $(2)$\\
    20 & 200 & $\bf{7.78}$ & $15$ & $\bf{0.961}$ & $\bf{0.961}$ & $\bf{43.8}$ & $44$ & $\bf{296}$ & $\bf{296}$\\
    &  & $(0.0105)$ & $(0.0386)$ &  &  &  &  & $(2)$ & $(2)$\\
    20 & 400 & $\bf{30}$ & $60.4$ & $\bf{1.42}$ & $1.45$ & $\bf{64.6}$ & $67.8$ & $377$ & $\bf{376}$\\
    &  & $(0.0157)$ & $(0.0841)$ &  &  &  &  & $(2)$ & $(2)$\\
    \bottomrule
    \end{tabular}
	\end{scriptsize}
    \caption{
        \label{lsADMM}
        Performance of ADMM on convex regression using CG and LSQR.
        Both inner and outer iterations are reported with the latter in parantheses.
    }
\end{table}

\section{Software \& Computing Environment}

Code for our implementations and numerical experiments is available at \url{https://github.com/alanderos91/ProximalDistanceAlgorithms.jl} and is based on the Julia language \citep{Julia-2017}.
Additional packages used include Plots.jl \citep{Plots.jl-2021}, GR.jl \citep{GR.jl-2021}, and \citep{Convex.jl-2014}. Numerical experiments were carried out on a Manjaro Linux 5.10.89-1 desktop environment using 8 cores on an Intel 10900KF at 4.9 GHz and 32 GB RAM.

\end{appendices}
%
%
\vskip 0.2in
\bibliography{references}

\end{document}